\documentclass[nopreprintline]{elsarticle}
\usepackage[colorlinks]{hyperref}
\AtBeginDocument{
	\hypersetup{
		colorlinks=true,
   		linkcolor=RoyalBlue,
    		citecolor=RedOrange,
 	 }
}
\usepackage{amsmath,amssymb,amsthm}
\usepackage{enumitem}
\usepackage[usenames,dvipsnames]{color}
\usepackage{pgf,tikz}
\usepackage{tkz-graph}
\usepackage{verbatim}
\usepackage{amsfonts}
\usepackage{mathrsfs}
\usepackage{amsbsy}
\usepackage{lineno} 
\usepackage{lipsum}
\makeatletter
\def\ps@pprintTitle{%
 \let\@oddhead\@empty
 \let\@evenhead\@empty
 \def\@oddfoot{}%
 \let\@evenfoot\@oddfoot}
\makeatother

\def\pr{\text{ P\/}}
\def\ex{\text{E\/}}

\def\eps{\varepsilon}

\def\part{\partial}
\def\Cal{\mathcal}
\def\var{\text{Var\/}}

\newcommand{\bs}{\boldsymbol}
\newcommand{\beq}{\begin{equation}}
\newcommand{\eeq}{\end{equation}}

\newcommand{\C}{\mathcal{CAT}}
\newcommand{\I}{\mathcal{INV}}

\newtheorem{Theorem}{Theorem}[section]
\newtheorem{Lemma}[Theorem]{Lemma}

\newtheorem{Corollary}[Theorem]{Corollary}

\theoremstyle{remark}

\newtheorem*{Remark}{\bf Remark}
\numberwithin{equation}{section}

\begin{document}

\begin{frontmatter}
\title{Formation of a giant component in the intersection graph of a random chord diagram}

\author{H\"{u}seyin Acan\fnref{1}}
\address{School of Mathematical Sciences, Monash University, VIC 3800, Australia}
\ead{huseyin.acan@monash.edu}
\author{Boris Pittel\fnref{1}}
\address{Department of Mathematics, The Ohio State University, Columbus, Ohio 43210, USA}
\ead{bgp@math.ohio-state.edu}
\fntext[1]{Research of both authors supported by NSF Grant DMS-1101237}

\begin{keyword}
chord diagram \sep enumeration \sep crossing \sep asymptotics  \sep giant component

05C30\sep 05C80 \sep 05C05\sep 34E05 \sep 60C05
\end{keyword}

\begin{abstract}
We study the number of chords and the number of crossings in the largest component of a random chord diagram when the chords are sparsely crossing. This is equivalent to studying the number of vertices and the number of edges in the largest component of the random intersection graph. Denoting the number of chords by $n$ and the number of crossings by $m$, when $m/n\log n$ tends to a limit in $(0,2/\pi^2)$, we show that the chord diagram chosen uniformly at random from all the diagrams with given parameters has a component containing almost all the crossings and a positive fraction of chords. On the other hand, when $m\le n/14$, the size of the largest component is of size $O(\log n)$. One of the key analytical ingredients is an asymptotic expression for the number of chord diagrams with parameters $n$ and $m$ for $m<(2/\pi^2)n\log n$, based
on the Touchard-Riordan formula and the Jacobi identity for the generating function of Euler partition function.
\end{abstract}

\end{frontmatter}

\section{Introduction}

A \emph{chord diagram} of size $n$ is a pairing of $2n$ points. It is customary to place the $2n$ points on a circle in general position, label them $1$ through $2n$ clockwise, and connect the two points in the pairing with a chord. Alternatively, we can represent a chord diagram by putting the numbers $\{1,\dots,2n\}$ on a line in increasing order and connecting the pairs of a chord diagram by an arc; we call it a \emph{linearized chord diagram}. For an illustration, see Figure~\ref{fig: CD and graph}.

\begin{figure}[t]
\begin{tabular} {c c}
\begin{tikzpicture}[scale = 1]

\coordinate (center) at (0,0);
\def\radius{1.2cm}
\draw (center) circle[radius=\radius];
\begin{scriptsize}
\foreach \x in {1,...,10}
{
\path (center) ++(180+0.1*180-\x*0.2*180:\radius) coordinate (A\x);
\fill (A\x) circle[radius=2pt] ++(180+0.1*180-\x*0.2*180:.7em) node {\x};
}
\foreach \from/\to in {A1/A4, A2/A5, A3/A6, A7/A9, A8/A10}
  \draw (\from) -- (\to);
\end{scriptsize}
\end{tikzpicture}

\begin{tikzpicture}[scale=1, line cap=round,line join=round,>=triangle 45,x=1.0cm,y=1.0cm]
\draw (0,0)--(6,0);
\foreach \x in {1,...,10}
{
\path (0.66*\x-.66,0)++ (0,0) coordinate (n\x);
\fill (n\x) circle[radius=2pt] ++ (0,-.3) node {\x};
}
\foreach \from/\to in {n1/n4, n2/n5, n3/n6, n7/n9, n8/n10}
\path (\from) edge[bend left=800] (\to);
\end{tikzpicture}
\end{tabular}
\caption{A circular and a linearized chord diagram. They are equivalent to each other.}
\label{fig: CD and graph}
\end{figure}
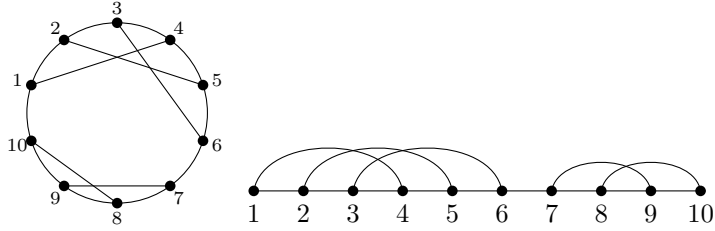

Chord diagrams appear in various contexts in mathematics, especially in topology. For instance, Chmutov and Duzhin \cite{CD}, Stoimenow \cite{Sto}, Bollob\'{a}s and Riordan \cite{BR2}, and Zagier \cite{Zagier} used chord diagrams to bound the dimension of the space of order $n$ Vassiliev invariants in knot theory. Rosenstiehl~\cite{Rosenstiehl} gave a characterization of Gauss words in terms of the intersection graphs of the chord diagrams.

As another application, consider an oriented surface obtained by taking a regular $2n$-gon and gluing the edges pairwise with opposite directions. Each such gluing defines a chord diagram; simply interpret the glued edges of the  $2n$-gon as pairs of endpoints of chords. In this 
topological context, it is natural to ask what the genus of a given chord diagram is. A remarkable formula for the generating function of the double sequence $c_g(n)$ was found by Harer and Zagier~\cite{HZ}; here $c_g(n)$ denotes the number of chord diagrams with $n$ chords and genus $g$. Linial and Nowik~\cite{LN} found the asymptotic likely  value of the genus of the chord diagram chosen uniformly at random. Subsequently, Chmutov and Pittel \cite{CP} proved that the genus of the random chord diagram is asymptotically Gaussian as $n$ tends to infinity. For detailed information about the chord diagrams and their topological and algebraic significance we refer the reader to Chmutov, Duzhin, and Mostovoy's book~\cite{CDM}.

A chord diagram of size $n$ can be thought of as a fixed-point-free involution of a set of $2n$ numbers. Baik and Reins \cite{Baik-Reins} found the aymptotic distribution of the length of the longest decreasing subsequence of a random fixed-point-free involution. Chen et al. \cite{Chen et al} showed that the crossing number and the nesting number of linearized chord diagrams have a symmetric joint distribution. Since the lack of a decreasing subsequence of length $2k+1$ in an involution is equivalent to the lack of $(k+1)$-nesting in the corresponding chord diagram, the result of Baik and Reins, combined with the result of Chen et al., gives the distribution for the maximum number of chords, all crossing each other, when the chord diagram is chosen uniformly at random.

In random graph theory, Bollob\'{a}s and O. Riordan~\cite{BR1} used random linearized chord diagrams to provide a precise description of the preferential attachment random graph model introduced by Barab\'{a}si and Albert~\cite{BA}.

It is easy to see that there are $(2n-1)!!$ chord diagrams of size $n$. However, enumerating
 chord diagrams with special properties could become hard rather quickly. A classic example is
counting  chord diagrams with a given number of crossings. This problem was first studied by Touchard~\cite{Touchard}, who found a bivariate generating function for $T_{n,m}$, the number of chord diagrams of size $n$ with $m$ crossings. Later, J. Riordan~\cite{Riordan} used Touchard's
formula to extract remarkable explicit formulas for $\sum_m q^m T_{n,m}$, and $T_{n,m}$ 
itself. However, the latter  is in the form of an alternating sum, indispensable for moderate values of $m$ and $n$, but not easily yielding an asymptotic approximation for $T_{n,m}$ for
$n,m\to\infty$. (We refer the reader to Aigner~\cite{Aig} for an eminently readable exposition
of the Touchard-Riordan achievement.)  A quarter century later, Flajolet and Noy~\cite{FN} were able to use J.  Riordan's formula for the univariate $\sum_m q^m T_{n,m}$ to
show that  the number of crossings in the uniformly random chord diagram is asymptotically Gaussian.  Cori and Marcus~\cite{CM} counted the number of isomorphism classes of  chord diagrams, with two chord diagrams being isomorphic if they are rotationally equivalent. 

Another way to represent a chord diagram $\Cal D$ is to associate with it a graph $G_{\Cal D}$,   called the \emph{intersection graph of $\Cal D$}. The vertices of $G_{\Cal D}$ are the chords of $\Cal D$ and there is an edge between two vertices in $G_{\Cal D}$ if and only if the corresponding chords cross each other in $\Cal D$, see Figure~\ref{graph}. If, instead of labeling the endpoints of the chords, we label the chords from $1$ to $n$ in an arbitrary way, we obtain a labeled \emph{circle graph}. Circle graphs are interesting in their own right and they have been studied widely.
A characterization of circle graphs was given by Bouchet \cite{Bouchet}. 
(Still, as Arratia et al.~\cite{Arratia} pointed out in a lucid discussion, even a formula, exact or asymptotic, for the number 
of circle graphs remains unknown.)

\begin{figure}[h]
\centering
\begin{tikzpicture}
\coordinate (n1) at (1,10) {};
\coordinate (n2) at (0,9) {};
\coordinate (n3)  at (2,9) {};
\coordinate (n4)   at (4,9) {};
\coordinate (n5) at (4,10) {};

\fill (n1) circle[radius=3pt] node[above] {(1,4)};
\fill (n2) circle[radius=3pt] node[left] {(2,5)};
\fill (n3) circle[radius=3pt] node[right] {(3,6)};
\fill (n4) circle[radius=3pt] node[right] {(7,9)};
\fill (n5) circle[radius=3pt] node[right] {(8,10)};

\foreach \from/\to in {n1/n2, n1/n3, n2/n3, n4/n5}
  \draw (\from) -- (\to);
\end{tikzpicture}
\caption{Intersection graph of the chord diagram given in Figure~\ref{fig: CD and graph}.}
\label{graph}
\end{figure}
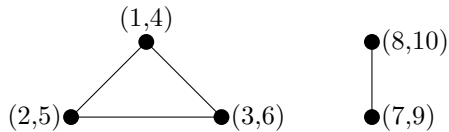

A chord diagram $\Cal D$ is \emph{connected} if there is no line cutting the circle 
that does not intersect any of the chords and partitions 
 the set of chords into two nonempty subsets. In other words, $\Cal D$ is connected if and only if $G_{\Cal D}$ is connected.
 
By making use of recurrence relations, Stein and Everett \cite{SE} proved that, as $n$ tends to infinity, the probability that a random chord diagram with $n$ chords is connected approaches $1/e$. Later, Flajolet and Noy \cite{FN} proved that almost all chord diagrams are \emph{monolithic}, i.e.,  consist of a single giant  component and a number of isolated chords. Having proved that in the limit the number of isolated chords was Poisson$(1)$, they recovered Stein and Everett's result. In his Ph.D. thesis \cite{Acan}, the first coauthor of this paper undertook an enumerative-probabilistic study of other parameters of chord diagrams and the associated intersection graphs, and in particular, extended the result of Flajolet and Noy in several directions. 

Our motivation for a probabilistic study of a random diagram comes from the realization that
its intersection graph represents a rather natural analogue of the classic random graph
$G(n,m)$, distributed uniformly on the set of all $\binom{\binom{n}{2}}{m}$ graphs on 
$[n]$ with $m$ edges. More than half a century ago, Erd\H{o}s and R\'{e}nyi~\cite{ER1}--\cite{ER2} basically created modern random graph theory by determining the sharp threshold
value of the number of edges in $G(n,m)$ for likely birth of
a giant component and the (larger) threshold value for the number of edges guaranteeing
that, with high probability (whp)\footnotemark, \footnotetext{A sequence of events $A_1,A_2,\dots$ occurs with high probability, abbreviated as whp throughout the paper, if $\lim_{n\to \infty}\pr(A_n)=1$.} the random graph is connected. What are then analogous 
thresholds for the intersection graphs of chord diagrams?

In this paper, we find some partial answers. We show that if the number of crossings $m=m(n)$
is such that  $\lim m/(n\log n) \in (0,2/\pi^2)$, then, with high probability  there is a giant component containing almost all  $m$ crossings and a positive fraction of all vertices. 

It is highly plausible that, for every $m>
(2/\pi^2)n\log n$, the intersection graph is likely to contain a giant component as well. 
To show this, presumably one has to find a way to ``embed''  the intersection graph with
$m_1$ crossings into that with $m_2$ crossings, whenever $m_1<m_2$. However, unlike the
Erd\H os-R\'enyi random graph $G(n,m)$, such an embedding is highly 
problematic, if possible at all, for the random intersection graphs. The only insight into
the component structure for those $m$'s  we have is a {\it gap property\/} for the number of
crossings in the
{\it densest\/} component, the one with the highest ratio of number of crossings to the number
of chords. Whp,  the number of crossings is either $O((m/n)\log n)$ or almost $m$, implying that the size of the densest component  is either at most $O((m/n)\log n)$, or at least $\sqrt{2m}$.

We also show that if $m\le n/14$, then
whp the largest component has a size below $5\log n/(\log \tfrac{225}{224})$.
The bound $m\le n/14$ may well be improved; by comparison, the giant-component threshold
for $G(n,m)$ is $m = n/2$,  see Erd\H{o}s-R\'{e}nyi~\cite{ER1}--\cite{ER2}, Bollob\'as~\cite{Bollobas}.

Now, for $G(n,m)$ the connectedness threshold is $m\sim (n \log n)/2$, (\cite{ER1}, \cite{ER2},
\cite{Bollobas}).  According to \cite{FN},
the crossing number of the uniformly random diagram is sharply concentrated
around its mean $\approx n^2/6$, with a standard deviation of order $n^{3/2}$, and the intersection graph is disconnected with 
positive limiting probability $1-e^{-1}$, see \cite{SE}. These results almost certainly rule out the existence of 
a connectedness threshold $m(n)=o(n^2)$. Still, we conjecture that $m(n)\approx n^{3/2}$ is
the threshold value of $m$ for the second largest component to be of bounded size. 

Among the key ingredients of our proofs is an asymptotic formula for $T_{n,m}$ for $m<(2/\pi^2)n \log n$, and a  bound $T_{n,m}\le C_n I_{n,m}$, where $C_n$ is the $n$-th Catalan number
and $I_{n,m}$ is the number of permutations of $[n]$ with $m$ inversions. The asymptotic formula
is based on the Touchard-Riordan sum-type formula, Jacobi's identity 
and Freiman's asymptotic formula for the generating function of the Euler partition
function. A final step in our argument is based on a rather deep formula for 
the number of non-crossing partitions with given block sizes, due to Kreweras~\cite{Kreweras}.

We should note that the Jacobi identity had appeared prominently in Josuat-Verg\'es and Kim's \cite{JK}
paper in the context of some new Touchard-Riordan type formulas for generating functions. 

The rest of the paper is organized as follows. In Section~\ref{sec: counting} we use the 
Touchard-Riordan formula to derive the asymptotic formula for $T_{n,m}$.  In Section~\ref{sec: bounds} we 
establish the upper and lower bounds for $T_{n,m}$. In Section~\ref{sec: large components} we use the asymptotics and the bounds for $T_{n,m}$ to prove our main results via analysis
of the likely sizes of the densest and the largest components in the random intersection graph. 
We conclude with a list of open problems.

\section{Counting moderately crossing chord diagrams}\label{sec: counting}
Let $T_{n,m}$ denote the number of
$n$-chord diagrams with $m$ crossings, where $m\in\bigl[0,\binom{n}{2}\bigr]$. For $n>0$, introduce the generating function $T_n(x)=\sum_m T_{n,m}x^m$, and let $T_0(x):=1$. Introduce
the bivariate generating function
\[
T(x,y)=\sum_{n\ge 0}T_n(x)y^n=\sum_{n,m}T_{n,m}x^my^n.
\]
Touchard discovered a remarkable formula for $T(x,y)$. To state it, introduce the
Catalan numbers $C_n=(n+1)^{-1} \binom{2n}{n}$, and the generating function of $\{C_n\}$,
\[
C(y)=\sum_{n\ge 0}C_ny^n.
\]
It is well known that the series converges for $|y|\le 1/4$, and for those $y$'s
\begin{equation}\label{Cquadr}
yC^2(y)-C(y)+1=0.
\end{equation}
Solving \eqref{Cquadr} with the initial condition $C(0)=C_0=1$, we find
\beq\label{Catalan solution}
C(y)= \frac{1-\sqrt{1-4y}}{2y}=\frac{2}{1+\sqrt{1-4y}}.
\eeq
For $D(y):=C(y)-1$, the equation \eqref{Cquadr} becomes
\begin{equation}\label{Dquadr}
D(y)=y(D(y)+1)^2.
\end{equation}
Then,  by Lagrange inversion formula,
\begin{equation}\label{Lag}
[y^n] D(y)^j =\frac{j}{n} [y^{n-j}](y+1)^{2n}=\frac{j}{n}\binom{2n}{n-j},\quad 0<j\le n;
\end{equation}
for $j=1$, we are back to $C_n=(n+1)^{-1}\binom{2n}{n}$.
Introduce also 
\begin{equation}\label{Axy=}
A(x,y)=\sum_{j\ge 0}x^{\binom{j+1}{2}}y^j.
\end{equation}
The series  converges for $|x|<1$ and all $y$. 
Touchard's formula states: for $|x|<1$, $|y|\le 1/4$,
\begin{equation}\label{Tou}
T(x,(1-x)y)=C(y) A(x,1-C(y)),\quad z:=\frac{y}{1-x}.
\end{equation}
Using Equation~\eqref{Tou}, Riordan found $T_{n,m}$ as an alternating sum.
\begin{Theorem}[Touchard-Riordan]
The number of chord diagrams with $n$ chords and $m$ crossings is given by
\begin{equation}\label{Rior}
T_{n,m}=\sum_{j}(-1)^j\binom{n+m-1-J(j)}{n-1}\frac{2j+1}{n+j+1}\binom{2n}{n-j},
\end{equation}
where $J(j)={j+1 \choose 2}$ and the sum is over all $j\ge 0$ such that $j\le n$, $J(j)\le m$.
\end{Theorem}
We prove the following more general statement.
\begin{Lemma}\label{xmynTell} Given $\ell\ge 1$, 
\begin{align}\label{xmynTell=sum}
&[x^my^n] T^{\ell}(x,y)=\sum_{\bs j=(j_1,\dots,j_{\ell})\ge \bs 0}
\binom{n+m-1-\sum_{\nu} J(j_{\nu})}{n-1} \prod_{\mu=1}^{\ell} (-1)^{j_{\mu}} \notag 	\\
&\quad\times\frac{2j+\ell}{2n+\ell}\binom{2n+\ell}{n-j},\quad j:=\sum_{\nu'}j_{\nu'},
\end{align}
and the sum is over $\bs j$  such that  $j\le n$ and $\sum_{\mu}J(j_{\mu})\le m$.
\end{Lemma}
\begin{proof}
Using \eqref{Dquadr}, \eqref{Lag}, and \eqref{Tou},
\begin{align}
&[x^my^n]T^{\ell}(x,y)=[x^my^n] \bigl((1-x)^{-n} C(y)^{\ell}A^{\ell}(x,1-C(y))\bigr)	\notag	\\
&=[x^my^n] \left\{(1-x)^{-n} (D(y)+1)^{\ell}\left(\sum_{j\ge 0}(-1)^j x^{J(j)}D(y)^j\right)^{\ell} \right\} 	\notag	\\
&=[x^my^n] \left\{(1-x)^{-n}\sum_{j_1,\dots,j_{\ell}\ge 0}\prod_{\mu=1}^{\ell}(-1)^{j_{\mu}}x^{J(j_{\mu})}
D(y)^{j_{\mu}}\left(\sum_{\kappa=0}^{\ell}\binom{\ell}{\kappa}D(y)^{\kappa}\right)\right\} 		\notag	\\
&=\sum_{j_1,\dots,j_{\ell}\ge 0}\prod_{\mu}(-1)^{j_{\mu}}\left([x^m] (1-x)^{-n}x^{\sum_{\nu}J(j_{\nu})}\right)\left(\sum_{\kappa=0}^{\ell}\binom{\ell}{\kappa}[y^n]D(y)^{\kappa +j}
\right)\notag\\
&=\sum_{j_1,\dots,j_{\ell}\ge 0} \left(\prod_{\mu=1}^{\ell} (-1)^{j_{\mu}}\right)
\binom{n+m-1-\sum_{\nu} J(j_{\nu})}{n-1} \notag\\
&\qquad \times\sum_{\kappa=0}^{\ell} \binom{\ell}{\kappa}\frac{\kappa+j}{n}\binom{2n}{n-\kappa-j},
\quad j:=\sum_{\nu'}j_{\nu'}.
\notag
\end{align}
Here
\[
\sum_{\kappa=0}^{\ell}\binom{\ell}{\kappa}\binom{2n}{n-\kappa-j}=\binom{2n+\ell}{n-j},
\]
and
\begin{align*}
\sum_{\kappa=0}^{\ell}\kappa\binom{\ell}{\kappa}\binom{2n}{n-\kappa-j}&=
\ell\sum_{r=0}^{\ell-1}\binom{\ell-1}{r}\binom{2n}{n-r-1-j}\\
&=\ell\binom{2n+\ell-1}{n-1-j},
\end{align*}
implying
\begin{align*}
\sum_{\kappa=0}^{\ell} \binom{\ell}{\kappa}\frac{\kappa+j}{n}\binom{2n}{n-\kappa-j}
&=\frac{j}{n}\binom{2n+\ell}{n-j}+\frac{\ell}{n}\binom{2n+\ell-1}{n-1-j}\\
&=\frac{2j+\ell}{2n+\ell}\binom{2n+\ell}{n-j}.
\end{align*}
This completes the proof of \eqref{xmynTell=sum}.
\end{proof}

Lemma \ref{xmynTell} enables us to derive an asymptotic formula for $[x^my^n] T^{\ell}(x,y)$,
whence for $T_{n,m}$,  when the number of crossings $m$ is not too large compared with $n$.
We begin with
\begin{Lemma}\label{4} Let $n\to \infty$ and $m=O(n)$.
Then, setting $q=m/(m+n)$,
\begin{equation}\label{asym*}
\begin{aligned}
T_{n,m}&\sim \binom{n+m-1}{n-1}C_n\prod_{j\ge 1}(1-q^j)^3,\\
[x^ny^m]T^{\ell}(x,y)&\sim \ell(2f(q))^{\ell-1}T_{n,m};\quad f(x):=\sum_{j\ge 0}(-1)^jx^{J(j)}.
\end{aligned}
\end{equation}
\end{Lemma}
\begin{proof} We notice upfront that $1-q$ is bounded away from $0$ for $m=O(n)$. Consider
the more difficult case $m\to\infty$.
Let $S_{n,m}(\bs j)$ denote the absolute value of the 
$\bs j$-th term of the sum in \eqref{xmynTell=sum}, i.e.,
\[
S_{n,m}(\bs j)=\binom{n+m-1-\Sigma(\bs j)}{n-1}\frac{2j+\ell}{2n+\ell}\binom{2n+\ell}{n-j},\quad
\Sigma(\bs j):=\sum_{\nu}J(j_{\nu}).
\]
Observe that
\begin{equation}\label{bin=bin}
\binom{n+m-1-\Sigma(\bs j)}{n-1}=\binom{n+m-1}{n-1}\prod_{k=m+1}^{n+m-1}\bigl(1-\Sigma(\bs j)/k\bigr),
\end{equation}
where $(a)_k$ denotes the $k$-th falling factorial of $a$.
Here
\begin{equation}\label{prod<}
\frac{(m)_{\Sigma(\bs j)}}{(m+n-1)_{\Sigma(\bs j)}}\le \left(\frac{m}{m+n-1}\right)^{\Sigma(\bs j)} \le  e q^{\Sigma(\bs j)},\quad (q=m/(m+n))
\end{equation}
since $\Sigma(\bs j)\le m$ for an admissible $\bs j$. Also
\begin{align*}
\frac{2j+\ell}{2n+\ell}\binom{2n+\ell}{n-j}\le \frac{2j+\ell}{2n+\ell} \binom{2n+\ell}{n}
\le \frac{2j+\ell}{2n+\ell} \,2^{\ell}\binom{2n}{n}\le 2^{\ell} (2j+\ell) C_n.
\end{align*}
Therefore, uniformly for all admissible $\bs j$,
\begin{equation}\label{Snmj,uniform}
S_{n,m}(\bs j)\le_b 2^{\ell}(2j+\ell)\binom{n+m-1}{n-1}C_n\cdot q^{\Sigma(\bs j)}.
\end{equation}
Here and elsewhere we use $A \le_b B$ as a shorthand for $A=O(B)$ when $B$ is too bulky.  Further,  $\bs j$ is certainly admissible if say $j<m^{1/5}$, and  it is not difficult to obtain 
that for those $\bs j$ 
\begin{equation}\label{Snmj,bounded}
S_{n,m}(\bs j)=\bigl(1+O(j^4/m)\bigr)
2^{\ell-1}(2j+\ell)\binom{n+m-1}{n-1}C_n\cdot q^{\Sigma(\bs j)}.
\end{equation}
Combining \eqref{Snmj,uniform} and \eqref{Snmj,bounded}, and using the uniform 
convergence of the infinite series $\sum_{\bs j\ge \bs 0} j^4 q^{\Sigma(\bs j)}$,  we get
\begin{multline}\label{comb}
[x^my^n]T^{\ell}(x,y)=\binom{n+m-1}{n-1}C_n2^{\ell-1}\\
\times\left[\sum_{\bs j}(2j+\ell)\prod_{\mu=1}^{\ell}(-1)^{j_{\mu}}q^{J(j_{\mu})}+o(1)\right],
\end{multline}
the sum being taken over {\it all\/} $\bs j\ge \bs 0$. Here
\[
\sum_{\bs j}\prod_{\mu=1}^{\ell}(-1)^{j_{\mu}}q^{J(j_{\mu})}=f(q)^{\ell},\quad
f(x)=\sum_{j\ge 0}(-1)^j x^{J(j)},
\]
and
\[
\sum_{\bs j} 2j\prod_{\mu=1}^{\ell}(-1)^{j_{\mu}}q^{J(j_{\mu})}=2\ell \left(\sum_{j_1\ge 0}
(-1)^{j_1} j_1q^{J(j_1)}\right) f(q)^{\ell-1}.
\]
So 
\begin{align*}
\sum_{\bs j}(2j+\ell)\prod_{\mu=1}^{\ell}(-1)^{j_{\mu}}q^{J(j_{\mu})}&=
\ell f(q)^{\ell-1}\left(f(q)+2\sum_{j_1\ge 0}(-1)^{j_1}j_1q^{J(j_1)}\right)\\
&=\ell f(q)^{\ell-1}\sum_{j_1\ge 0}(-1)^{j_1}(2j_1+1)q^{J(j_1)},
\end{align*}
and \eqref{comb} becomes
\begin{multline}\label{sum2j+1}
[x^my^n]T^{\ell}(x,y)=\binom{n+m-1}{n-1}C_n\ell \\
\times \left[(2f(q))^{\ell-1}\sum_{j\ge 0}(-1)^j (2j+1)
q^{J(j)}+o(1)\right].
\end{multline}
Since the series for $f(x)$ is alternating, and $q^{J(j)}\downarrow 0$, we have $f(q)>1-q$,
i.e., $f(q)$ is bounded away from zero.  However, $(2j+1)q^{J(j)}$ is not monotone, and bounding  the last alternating sum from below would be a rather hard task. 
Fortunately,
there is a remarkable identity discovered by Jacobi as a corollary of the classic triple
product identity, Andrews et al. \cite[Page 500]{And}:
\begin{equation}\label{Jacobi}
\sum_{j\ge 0}(-1)^j (2j+1) x^{J(j)}= \prod_{j\ge 1}(1-x^j)^3,\quad |x|<1.
\end{equation}
Since our $q=m/(m+n)$ is bounded away from $1$ for $m=O(n)$, the product on the RHS
of \eqref{Jacobi} for $x=q$ is bounded away from zero uniformly for $n$.  The equations 
\eqref{sum2j+1} and \eqref{Jacobi} complete the proof of Lemma \ref{4}.
\end{proof}

The reader is correct to suspect that the constraint $m=O(n)$ is unnecessarily restrictive. In our
next statement we extend the asymptotic formulas to $m\le c n \log n$.  We hope that
the inevitably more technical argument can be understood more easily since we will use
the proof above as a rough template. 

\begin{Lemma}\label{ell components} 
Let  $\ell\ge 1$ be given. If
\begin{equation}\label{m leq}
m\le \frac{2}{\pi^2}n\bigl(\log n - 0.5(\ell+2)\log\log n -\omega(n)\bigr),
\end{equation}
where $\omega(n)\to\infty$ however slowly, then 
\begin{equation}\label{asym,ext,ell}
[x^ny^m] \,T(x,y)^{\ell}\sim \ell (2f(q))^{\ell-1} \binom{n+m-1}{n-1}C_n\prod_{j\ge 1}(1-q^j)^3.
\end{equation}
\end{Lemma}

\begin{proof} It suffices to consider the case $m/n\to\infty$, in which case $q\to 1$. We still have
$f(q)>1-q>0$, but $1-q\to 0$. However, by Tauberian theorem  for power series,
$f(1-) = 1/2$, whence $\liminf f(q)=1/2>0$.

As in the proof of Lemma \ref{4}, our starting point is the identity
\begin{align*}
&[x^my^n] T^{\ell}(x,y)=\sum_{\bs j=(j_1,\dots,j_{\ell})\ge \bs 0}
\binom{n+m-1-\Sigma(\bs j)}{n-1} \prod_{\mu=1}^{\ell} (-1)^{j_{\mu}}\\
&\quad\times\frac{2j+\ell}{2n+\ell}\binom{2n+\ell}{n-j},\quad j:=\sum_{\nu'}j_{\nu'},
\end{align*}
where $\Sigma(\bs j)=\sum_{\nu}J(j_{\nu})$ and $j=\sum_{\nu}j_{\nu}$, and we focus on $S_{n,m}(\bs j)$, the absolute
value of the $\bs j$-th summand in the sum. The uniform bound  \eqref{Snmj,uniform} 
continues to hold. 
Setting $M:=\lfloor a (1-q)^{-1}\rfloor$ for some $a>1$, we write the sum as $S_1+S_2$, where $S_1$ is the contribution of $\bs j$'s with $\max_i j_i \le M$ and $S_2$ is 
the contribution of the remaining $\bs j$'s. For the terms in $S_1$, analogously to 
\eqref{Snmj,bounded} we have
\[
S_{n,m}(\bs j)=\bigl(1+O(\Sigma(\bs j)^2/m)\bigr) 2^{\ell-1}(2j+\ell)\binom{n+m-1}{n-1}C_n\cdot 
q^{\Sigma(\bs j)}.
\]
Therefore, $S_1=S_{11}+R_1$, where
\begin{align}\label{asymp S_1, 1}
S_{11}&= \binom{n+m-1}{n-1}C_n 2^{\ell-1}
\sum_{j_1,\dots,j_{\ell}\atop \max j_i\le M}\,\,\left(\sum_{t =1}^{\ell}(2j_t+1)\right)\prod_{\mu=1}^{\ell} (-1)^{j_{\mu}} q^{J(j_{\mu})}				\\
&= \binom{n+m-1}{n-1}C_n \ell\, \sum_{j=0}^{M}(-1)^j (2j+1)q^{J(j)} \left(2\sum_{k=0}^{M}(-1)^k q^{J(k)}\right)^{\ell-1}, \notag
\end{align}
and 
\begin{equation}\label{|R1|<}
|R_1|\le_b \frac{1}{m}\binom{n+m-1}{n-1}C_n 
\left(\sum_{j\ge 0}j^5q^{J(j)}\right)
\left(\sum_{k\ge 0}q^{J(k)}\right)^{\ell-1}.
\end{equation}
For the last bound we have used $\Sigma(\bs j)^2(\sum_{\mu}j_{\mu})\le_b \sum_{\nu}j_{\nu}^5$ and the fact that $\ell$ is fixed.
Defining the functions
\begin{equation*}
h_M(q)=  \sum_{j=M+1}^{\infty}(-1)^j (2j+1)q^{J(j)}; \quad 
f_M(q)=\sum_{k=0}^{M}(-1)^k q^{J(k)},
\end{equation*}
and using \eqref{Jacobi}  on the last line of \eqref{asymp S_1, 1}, we write
\begin{equation}\label{S11,M}
S_{11}= \binom{n+m-1}{n-1}C_n \ell\,  \left( \prod_{j\ge 1} (1-q^j)^3 +h_M(q)\right) (2f_M(q))^{\ell-1}. 
\end{equation}
Now $q^{J(j)}\le \exp(-j^2(1-q)/2)$, and $x \exp(-x^2(1-q)/2)$ attains its maximum at 
$(1-q)^{-1/2}\ll a(1-q)^{-1}=M$. So 
\[
|h_M(q)| \le_b \int_M^{\infty} x\exp(-x^2(1-q)/2)=(1-q)^{-1}\exp\left(-\frac{a^2}{2(1-q)}\right).
\]
Also,
\[
\bigg| \sum_{k>M} (-1)^kq^{J(k)}\bigg| \le q^{J(M+1)} \le q^{a^2(1-q)^{-2}/2} \le \exp 
\left(-\frac{a^2}{2(1-q)} \right),
\]
where the last inequality follows from $x^{1/(1-x)}\le e^{-1}$ for any $x \in (0,1)$. So,
using Jacobi identity, 
\begin{multline*}
\sum_{0\le j\le M}(-1)^j (2j+1)q^{J(j)} =  \prod_{j\ge 1}(1-q^j)^3 +O\bigl((1-q)^{-1}e^{-a^2/2(1-q)}\bigr).
\end{multline*}
Here,  by Freiman's asymptotic formula (see Postnikov~\cite[Sect. 2.7]{Post}, and also
Pittel~\cite[Eq. 2.8]{Pittel}, \cite[Sect. 2]{Pittel'}),
\begin{equation}\label{Freiman}
\begin{aligned}
\prod_{j\ge 1}(1-q^j)=&\,\exp\left(-\frac{\pi^2}{6z}-\tfrac{1}{2}\log\frac{z}{2\pi}+O(|z|)\right)_{z=
-\log q}\\
=&\,\exp\left(-\frac{\pi^2}{6(1-q)}-\frac{1}{2}\log(1-q)+O(1)\right).
\end{aligned}
\end{equation}
Therefore
\[
\sum_{0\le j\le M}(-1)^j (2j+1)q^{J(j)}=\prod_{j\ge 1}(1-q^j)^3\left(1+O\bigl((1-q)^{1/2}e^{-(a^2-\pi^2)/2(1-q)}\bigr)\right).
\]
Also,  using $\lim f(q)=1/2>0$, 
\[
f_M(q) =f(q)- \sum_{k> M}(-1)^k q^{J(k)}= f(q)\left(1+O\bigl(e^{-a^2/2(1-q)}\bigr)\right).
\]
So, selecting $a=\pi \sqrt{3}$ say, \eqref{asymp S_1, 1} becomes 
\begin{equation}\label{S_1,1,explicit}
S_{11}= \left( 1+  O\bigl(e^{-(1-q)^{-1}}\bigr) \right)\ell \bigl(2 f(q)\bigr)^{\ell-1}
\binom{n+m-1}{n-1}C_n \prod_{j\ge 1}(1-q^j)^3.
\end{equation}
Furthermore, \eqref{|R1|<} together with the bounds 
\begin{equation}\label{bounds}
\sum_{j\ge 0}j^5q^{J(j)}=O((1-q)^{-3}),\quad 2\sum_{k\ge 0}q^{J(k)} \le 2(1-q)^{-1/2}
\end{equation}
yield
\begin{equation}\label{R11,explicit}
|R_1|\le_b \frac{1}{m}(1-q)^{-(\ell+5)/2}\binom{n+m-1}{n-1}C_n.
\end{equation}
Using Freiman's formula and the condition \eqref{m leq}, we have
\begin{equation*}
\begin{aligned}
\frac{m^{-1}(1-q)^{-(\ell+5)/2}}{\prod_{j\ge 1}(1-q^j)^3}&\le_b\exp(\pi^2m/2n)\,\frac{m^{(\ell+3)/2}} {n^{(\ell+5)/2}}\\
&\le_b (\log n)^{(\ell+3)/2}n^{-1}\exp(\pi^2m/2n-0.5\log\log n)\\
&\le_b\exp\left(\frac{\ell+2}{2}\log\log n-\log n+\frac{\pi^2m}{2n}\right)\\
&\le e^{-\omega(n)}\to 0.
\end{aligned}
\end{equation*}
So it follows from \eqref{S_1,1,explicit}  and \eqref{R11,explicit} that
\begin{equation}\label{S1approx}
S_1 =(1+o(1))\binom{n+m-1}{n-1}C_n \ell (2f(q))^{\ell-1}.
\end{equation}
It remains to show that $S_2$ is negligible compared to $S_1$.
It is easy to see that
\begin{align*}
|S_2| &\le_b \binom{n+m-1}{n-1}C_n\sum_{j_1,\dots,j_{\ell}\atop \max j_i> M}\, \prod_{\mu} q^{J(j_{\mu})}\sum_{t =1}^{\ell}(2j_t+1) \\
&\le_b \binom{n+m-1}{n-1}C_n\, \left(\sum_{j\ge 0}q^{J(j)}\right)^{\ell-2}\\
&\times\left[\sum_{j_1>M}(2j_1+1)q^{J(j_1)}\sum_{j\ge 0}q^{J(j)}
+\sum_{j_1\ge M}q^{J(j_1)} \sum_{j\ge 0} (2j+1)q^{J(j)}\right].
\end{align*}
Here
\begin{align*}
&\sum_{j_1\ge M}q^{J(j_1)}\le_b (1-q)^{-1/2}\exp\left( -\tfrac{a^2}{2(1-q)}\right),\\
&\sum_{j_1\ge M}(2j_1+1)q^{J(j_1)}\le_b (1-q)^{-3/2}\exp\left( -\tfrac{a^2}{2(1-q)}\right).
\end{align*}
Combining these bounds with \eqref{bounds}, we obtain
\begin{equation}\label{|S2|<}
|S_2|\le_b \binom{n+m-1}{n-1}C_n (1-q)^{-(\ell+2)/2}\exp\left( -\tfrac{a^2}{2(1-q)}\right).
\end{equation}
Since
\[
(1-q)^{-(\ell+2)/2}\exp\left( -\tfrac{a^2}{2(1-q)}\right)\le_b \left(\frac{m}{n}\right)^{(\ell+2)/2}
\exp\left(-\frac{a^2m}{2n}\right)\to 0,
\]
it follows from \eqref{S1approx} and \eqref{|S2|<} that $S_1\gg |S_2|$. This finishes the proof.
\end{proof}
\begin{Remark} Whether the constraint \eqref{m leq} can be relaxed to, say, $m=\Theta(n \log n)$
is, in our opinion, a hard open problem.
\end{Remark}
Next, we apply Lemma \ref{4} to find the number of cuts in a random linearized chord diagram. A \emph{cut} is a partition of $[2n]$ into two blocks $[2n_1]$ and $[2n]\setminus [2n_1]$ such that there is no chord joining two points from different blocks. Let $X_{n,m}$ be the random variable counting the cuts in the linearized chord diagram chosen uniformly at random among all diagrams with $m$ crossings. 
Notice upfront that
\[
\pr(X_{n,m}\ge 1)\ge \frac{T_{n-1,m}}{T_{n,m}},
\]
as $T_{n-1,m}$ counts the linearized chord diagrams with an arc from the point
$1$ to the point $2$. Therefore, for $m=O(n)$, Lemma \ref{4} implies that
\[
\lim\inf \pr(X_{n,m}\ge 1)\ge \liminf \frac{C_{n-1}}{C_n}
\cdot\frac{\binom{n+m-2}{n-2}}{\binom{n+m-1}{n-1}}=\liminf\frac{n-1}{4(n+m-1)}>0.
\]

\begin{Theorem}\label{thm: cuts} Suppose $m=O(n)$. For each $j\ge 0$,
\begin{equation}\label{XnmtoX}
 \pr(X_{n,m}=j)=(j+1)(1-p)^2p^j+o(1);\quad p=1-(2f(q))^{-1},
\end{equation}
where $f(q)>1/2$ and bounded away from $1/2$.
\end{Theorem}

\begin{Remark} Using a Tauberian theorem one can show that
\[
f(x) = 1/2 +O((1-x)^{1/2}),\quad x\uparrow 1.
\]
Therefore, if $m/n$ is large then $f(q)$ is close to $1/2$, whence $p$ is close to zero.  
Consequently, whp, there is no cut in the random linearized chord diagram when $m/n$ tends to infinity. At the other extreme, $p=1/2$ for $m=0$, whence
\[
\lim_{n\to\infty} \pr(X_{n,0}=j)=(j+1)2^{-(j+2)},\quad j\ge 0.
\]
Lastly, a byproduct of this Theorem is a pure-calculus inequality $f(x)>1/2$ for $x\in [0,1)$, which seems hard to prove out of the context of the chord diagrams.
\end{Remark}

\begin{proof}[Proof of Theorem \ref{thm: cuts}]
Observe that
\begin{align*}
\ex\left[\binom{X_{n,m}}{k}\right]=&\,\frac{1}{T_{n,m}}\sum_{(n_1,m_1),\dots,(n_{k+1},m_{k+1})\atop
\sum_in_i =n,\,\sum_jm_j=m;\,\, n_1,\dots, n_{k+1}>0}
\prod_{i=1}^{k+1} T_{n_i,m_i}\\
=&\,\frac{1}{T_{n,m}}\,[x^my^n] (T(x,y)-1)^{k+1}\\
=&\,\frac{1}{T_{n,m}}\sum_{\ell =0}^{k+1}(-1)^{k+1-\ell}\binom{k+1}{\ell} [x^my^n] (T(x,y))^{\ell}\\
=&\,\sum_{\ell=0}^{k+1}(-1)^{k+1-\ell}\binom{k+1}{\ell}\frac{T_{n,m}^{(\ell)}}{T_{n,m}}.
\end{align*}
So, by Lemma \ref{4},
\begin{align*}
\ex\left[\binom{X_{n,m}}{k}\right]=&\, o(1)+\sum_{\ell=0}^{k+1}(-1)^{k+1-\ell}\ell\binom{k+1}{\ell}
(2f(q))^{\ell-1}\\
=&\,(k+1)(2f(q)-1)^k+o(1).
\end{align*}

In particular, it follows that $\liminf f(q)\ge 1/2$. If $\liminf f(q)=1/2$ then $\ex[X_{n.m}]\to 0$
and $\pr(X_{n,m}>0)\to 0$, which we ruled out earlier. Thus $\liminf f(q)>1/2$ if $m=O(n)$,
which effectively proves that $f(x)>1/2$ for all $x\in [0,1)$. 
By the last equation, $X_{n,m}$ is asymptotic, in distribution, to $X=X_{m/n}$, such that
\[
\ex\left[\binom{X}{k}\right]=(k+1)(2f(q)-1)^k.
\]
Notice that, for $z>0$ small enough,
\begin{align*}
\ex\bigl[z^X\bigr]=&\,\ex\left[(1+(z-1))^X\right]=\sum_{k\ge 0}(z-1)^k\ex\left[\binom{X}{k}\right]\\
=&\,\sum_{k\ge 0}(z-1)^k(k+1)(2f(q)-1)^k\\
=&\,\frac{1}{\bigl[1-(z-1)(2f(q)-1)\bigr]^2}\\
=&\,\left(\frac{1-p}{1-zp}\right)^2;\quad p=p(q):=1-(2f(q))^{-1};
\end{align*}
($p(q)\in (0,1)$ since $f(q)>1/2$ for $q<1$). Therefore, $X\overset{\mathcal D}\equiv Y^\prime + Y^{\prime\prime}$, where $Y^\prime$ and  $Y^{\prime
\prime}$ are independent copies of the geometric $Y$, 
\[
\pr(Y=j)=(1-p) p^j,\quad j\ge 0.
\]
Finally,
\begin{align*}
\left(\frac{1-p}{1-zp}\right)^2=&\,(1-p)^2\sum_{j\ge 0} (-1)^j\binom{-2}{j}p^j z^j\\
=&\,(1-p)^2\sum_{j\ge 0}(j+1)p^jz^j, \qedhere
\end{align*}
which implies that 
\[
\pr(X_{n,m}=j)=o(1)+\pr(X=j)=(1-p)^2 (j+1)p^j +o(1),\quad j\ge 0.
\]
\end{proof}

\begin{Remark} Intuitively, whp all the cuts are relatively close to the point $1$ or point
$2n$, and the numbers of those ``left'' and ``right'' cuts are asymptotically independent,
each close to the geometric $Y$. We will prove the first part of this conjecture  in the next
section, as an application of an upper bound for $T_{\nu,\mu}$ that holds for all values of
the parameters $\nu$ and $\mu$.
\end{Remark}

\section{Bounds for $T_{n,m}$}\label{sec: bounds}
In this section we will give some bounds on $T_{n,m}$. To this end, we need to introduce
another double-index sequence $\{I_{n,m}\}$, where $I_{n,m}$ denotes the number 
of permutations of $[n]$ with $m$ inversions. Each of those $I_{n,m}$ permutations 
$p=(p_1,\dots,p_n)$ gives
rise to an \emph{inversion sequence} $\bs x=(x_1,\dots,x_n)$: $x_j$ is the number of pairs $(p_i,p_j)$ such that
$i<j$ and $p_i>p_j$. Obviously $x_i\le i-1$ and $\sum_ix_i =m$. Conversely, every such
sequence $\bs x$ determines a unique permutation $p$ such that $\bs x$ is $p$'s inversion
sequence. Existence of this bijective correspondence implies a classic identity
\begin{equation}\label{Inm=[.]}
I_{n,m}=[z^m]\prod_{j=0}^{n-1} (1+z+\cdots+z^j)=[z^m] \prod_{j=1}^{n} \frac{1-z^j}{1-z}.
\end{equation}
Clearly, $I_{n,m}$ is at most the the number of nonnegative integer solutions of the equation
$x_1+\cdots+x_n=m$, i.e., 
\begin{equation}\label{upper I_n,m}
I_{n,m} \le {n+m-1 \choose n-1}.
\end{equation}

\begin{Lemma}\label{lem: asym IS}
Suppose that $m/n \to \infty$ and $m=o\bigl(n^{3/2}\bigr)$. Then,
\[
I_{n,m} \ge {n+m-1 \choose n-1} \exp\left( -\frac{\pi^2m}{6n} +O(\log n)\right).
\]
\end{Lemma}
\begin{proof}
 Pick $\rho\in (0,1)$, and introduce the sequence $\bs Y=(Y_1,\dots,Y_n)$ of {\it independent\/} random variables
such that 
\[
\pr(Y_i=j)=\frac{(1-\rho)\rho^j}{1-\rho^i},\quad 0\le j\le i-1;
\]
so
\[
\ex\bigl[z^{Y_i}\bigr]=\frac{1-\rho}{1-\rho^i}\cdot\frac{1-(\rho z)^{i}}{1-\rho z},\quad 1\le i\le n.
\]
Then \eqref{Inm=[.]} becomes
\begin{align}
I_{n,m}&=\rho^{-m}[z^m]\prod_{i=1}^n\frac{1-(\rho z)^i}{1-\rho z}
=\rho^{-m}\prod_{k=1}^n\frac{1-\rho^i}{1-\rho}\,[z^m]\prod_{i=1}^n\ex\bigl[z^{Y_i}\bigr]\notag\\
&=\rho^{-m}\prod_{k=1}^n\frac{1-\rho^i}{1-\rho}\,\pr(\|\bs Y\|=m),\label{rhobound}
\end{align}
where $\|\bs Y\|:=\sum_iY_i$. In particular,
\begin{equation}\label{Inm,Cher}
I_{n,m}\le \rho^{-m}\prod_{k=1}^n\frac{1-\rho^i}{1-\rho},
\end{equation}
and we get the best upper bound by selecting $\rho^*$ that minimizes the right side of \eqref{Inm,Cher}. Since $I_{n,m}$
does not depend on $\rho$, this $\rho^*$ maximizes $\pr(\|\bs Y\|=m)$. What remains
is to prove the existence of $\rho^*$ and to find an asymptotic formula for  that probability, i.e., to prove a local 
limit theorem for $||\bs Y||$. 

Crucially, the distribution of $\sum_i Y_i$ is log-concave, i. e.,
\begin{equation}\label{log-conc,def}
\pr(\|\bs Y\|=j)^2\ge \pr(\|\bs Y\|=j-1)\pr(\|\bs Y\|=j+1),\quad j\ge 0.
\end{equation}
The reason is that 
each $Y_i$ has a log-concave distribution and the convolution of log-concave distributions
is log-concave as well, Menon~\cite{Menon}. Even stronger, in terminology of Canfield
\cite{Canfieldloc}, the distribution of $\|\bs Y\|$ is {\it properly\/} log-concave, meaning that (a) the
range of $\|\bs Y\|$ has no gaps, and (b)
the equality in \eqref{log-conc,def} holds only if $\pr(\|\bs Y\|=j)=0$. 

Indeed, $Y_1=0$ is 
properly log-concave distributed, and (induction step)  proper log-concavity of $Z_{s+1}:=\sum_{r=1}^{s+1}Y_r$ for $s\ge 1$ follows from proper log-concavity of $Z_s:=\sum_{r=1}^{s}Y_r$ and the identity \cite{Canfieldloc}
\begin{multline}\label{CanIdentity}
P_{s+1,\nu}^2 -P_{s+1,\nu-1}P_{s+1,\nu+1}
=\sum_{\alpha<\beta}\bigl(P_{s,\alpha}P_{s,\beta-1}-P_{s,\alpha-1}P_{s,\beta}\bigr)
\\
\times \bigl(p_{s+1,\nu-\alpha}\,p_{s+1,\nu-\beta+1}-p_{s+1,\nu-\alpha+1}\,
p_{s+1,\nu-\beta}\bigr),
\end{multline}
$P_{t,\mu}:=\pr(Z_t=\mu)$, $p_{t,\mu}:=\pr(Y_t=\mu)$. Here is how.  Each summand on the RHS of
\eqref{CanIdentity} is non-negative as both $Z_s$ and $Y_{s+1}$ are log-concave,
and their respective ranges,   $[0,1,\dots, \binom{s}{2}]$ and  $[0,1,\dots,s]$, have no gaps.
If $\nu\le \binom{s}{2}$, we see that the summand for $\alpha=\nu$, $\beta=\nu+1$ is
\[
(P_{s,\nu}^2 - P_{s,\nu-1}P_{s,\nu+1})p_{s+1,0}^2>0,
\]
because, by inductive hypothesis,  $\{P_{s,t}\}$ is properly log-concave and $P_{s,\nu}>0$. If $\binom{s}{2}<\nu\le 
\binom{s+1}{2}$, we consider $\alpha=\nu-s$ and $\beta=\alpha+1$. Then
\begin{align*}
\alpha=\nu-s&\le \binom{s+1}{2}-s=\binom{s}{2},\\
\alpha=\nu-s&\ge \binom{s}{2}+1-s=\binom{s-1}{2}\ge 0,
\end{align*}
whence $P_{s,\alpha}>0$. Then the corresponding summand on the RHS of \eqref{CanIdentity} is
\[
(P_{s,\alpha}^2 - P_{s,\alpha-1}P_{s,\alpha +1})(p_{s+1,s}^2-p_{s+1,s+1}p_{s+1,s-1})>0,
\]
because, by inductive hypothesis, $P_{s,\alpha}^2 - P_{s,\alpha-1}P_{s,\alpha +1}>0$, and  $p_{s+1,s}>0$, $p_{s+1,s+1}=0$.

\begin{Remark} In Canfield \cite{Canfieldloc} the striking identity \eqref{CanIdentity} was used to show that
the convolution operation preserves the proper log-concavity. We had to use this identity
differently, i.e.\  inductively, because none of $Y_3,\dots,Y_n$ is {\it properly\/} log-concave. Notice also that for $\rho = 1$ our claim reduces to proper logconcavity
of $I_{n,m}$ for every $n\ge1$. The usual logconcavity of this sequence is long known, of course.
More recently B\'{o}na~\cite{Bona05} found a purely combinatorial proof of this property,
a proof that does not rely on Menon's theorem.
\end{Remark}

For $x\in (0,1)$, introduce 
\begin{align*}
L(x):&=\log \left(x^{-m}\prod_{k=1}^n\frac{1-x^k}{1-x}\right) \\
&=-m\log x +\sum_{k=1}(\log (1-x^k) - \log (1-x)).
\end{align*}
The stationary points of $L(x)$ are the roots, if any exist, of
\[
L^\prime(x)=\frac{1}{x}\left(n\frac{x}{1-x} -m -\sum_i\frac{ix^i}{1-x^i}\right)=0.
\]
Since the function within the braces is strictly increasing, there can be at most one stationary point,
and it is necessarily the point where $L(x)$ attains its minimum. Pick a constant $A>0$ and
introduce $\rho=q(1+A/n)$, $q=m/(m+n)$. Since $m/n^2\to 0$,
\[
1-\rho =(1-q)\bigl(1+O(mn^{-2})\bigr),
\]
whence
\[
n\frac{\rho}{1-\rho}- m=A\left(\frac{m}{n}\right)^2 +O(m/n).
\]
Further, approximating the sum $\sum_i i\rho^i/(1-\rho^i)$ by the corresponding integral we obtain
\begin{align*}
\sum_i\frac{ix^i}{1-x^i}&=\frac{1}{(\log (1/\log\rho))^2}\int_0^{\infty}\frac{x}{e^x-1}\,dx+O((1-\rho)^{-1})\\
&=\frac{\pi^2}{6}\left(\frac{m}{n}\right)^2 +O(m/n).
\end{align*}
Therefore, for $n$ large enough, $L^\prime(\rho)<0$ for $A<\pi^2/6$ and $L^\prime(\rho)>0$
if $A> \pi^2/6$. Thus the equation $L^\prime(x)=0$ does have a root $\rho^*$, and
$\rho^* = q (1+O(n^{-1})$. Furthermore,  uniformly for $x$ between $\rho:=\rho^*$ and $q$, 
\begin{align*}
L^{\prime\prime}(x) &=\frac{m}{x^2}+\frac{n}{(1-x)^2}+\sum_i\left(\frac{ix^{i-2}}{1-x^i}-\frac{i^2x^i}
{(1-x^i)^2}\right)\\
&=O(m+m^2/n+m^3/n^3)=O(m^2/n).
\end{align*}
Therefore
\begin{equation*}
L(\rho) = L(q) +O\bigl(m^2n^{-1}(\rho-q)^2\bigr)=L(q)+O(m^2/n^3)=L(q)+o(1),
\end{equation*}
since $m=o(n^{3/2})$.

By independence of $Y_1,\dots,Y_n$, using Berry-Esseen inequality (Feller~\cite{Feller},
Ch. XVI, Section 5),
\[
\max_{x\in \mathbb R}\left|\pr\left(\|\bs Y\| \le  \ex[\|\bs Y\|] + x\sigma(\|\bs Y\|)\right)-
\frac{1}{\sqrt{2\pi}}\int_{-\infty}^xe^{-y^2/2}\,dy\right|\le 6\frac{r_3}{\sigma^3},
\]
where
\begin{align*}
\sigma^2 =&\var(\|\bs Y\|)=\sum_{i=1}^n\var(Y_i)=\sum_{i=1}^n\ex\bigl[(Y_i-\ex[Y_i])^2\bigr],\\
r_3=&\,\sum_{i=1}^n\ex\bigl[|Y_i-\ex[Y_i]|^3\bigr].
\end{align*}
To compute $\sigma^2$ we use 
\[
\ex\bigl[z^{\|\bs Y\|}\bigr] =\prod_{i=1}^n \ex\bigl[z^{Y_i}\bigr]=\prod_{i=1}^n
\frac{1-\rho}{1-\rho^i}\cdot\frac{1-(\rho z)^i}{1-\rho z},
\]
and
\[
\left.\frac{d^2}{dz^2}\ex\bigl[z^{\|\bs Y\|}\bigr]\right|_{z=1}=\ex\bigl[(\|\bs Y\|)_2\bigr].
\]
Computing the derivative and bounding the resulting sum by the integral we obtain
\begin{align*}
\sigma^2=&\,\ex\bigl[(\|\bs Y\|)_2\bigr] +\ex[\|\bs Y\|] - \bigl(\ex[\|\bs Y\|] \bigr)^2\\
=&\,n\,\frac{\rho}{(1-\rho)^2}-\sum_{i=1}^n\frac{i^2\rho^i}{(1-\rho^i)^2}\\
=&\,(1+o(1))\,\frac{m^2}{n} - \Theta((m/n)^3)
=\,(1+o(1))\,\frac{m^2}{n}.
\end{align*}
Similar, but more protracted, computations lead to
\begin{align*}
r_4:=&\,\sum_{i=1}^n \ex\bigl[(Y_i-\ex[Y_i])^4\bigr]=(1+o(1))\frac{n}{(1-\rho)^4}\\
=&\,(1+o(1))\,\frac{m^4}{n^3}.
\end{align*}
Therefore,
\[
r_3\le n^{1/4} (r_4)^{3/4}=(1+o(1))\frac{m^3}{n^2}.
\]
Consequently, for $n$ large enough,
\begin{equation}\label{Berry}
\max_{x\in \Bbb R}\left|\pr\left(\|\bs Y\| \le  \ex[\|\bs Y\|] + x\sigma(\|\bs Y\|)\right)-
\frac{1}{\sqrt{2\pi}}\int_{-\infty}^xe^{-y^2/2}\,dy\right|\le 7n^{-1/2}.
\end{equation}
If we write $7n^{-1/2}=K/\sigma$ then 
\[
\frac{K}{\sqrt{\sigma}}=7\sqrt{\frac{\sigma}{n}}\le 8\sqrt{\frac{m}{n^{3/2}}}\to 0,
\]
since $m\ll n^{3/2}$. Applying Canfield's quantified version of Bender's local limit theorem for
properly log-concave distributions obeying Berry-Esseen estimate, (Bender~\cite{Benderloc}, 
Canfield~\cite{Canfieldloc}), we conclude that
\begin{equation}\label{pr(Y=m)=}
\pr(\|\bs Y\|=m)=\frac{1+o(1)}{\sqrt{2\pi \var(\|\bs Y\|)}}=\Theta(n^{1/2}m^{-1}).
\end{equation}
Next,
by Lemma 3.11 in \cite{AP}, we have
\[
\prod_{j=1}^{n}(1-q^j)= \left( 1+ O(m^2/n^3) \right)\prod_{j=1}^{\infty}(1-q^j),
\]
and, by the proof of Lemma 3.14 in \cite{AP}, we have
\[
\prod_{j=1}^{\infty}(1-q^j) \sim  K\cdot \exp\left( -(\pi^2/6)(m/n)+(1/2)\log(m/n)\right),
\]
where $K= \sqrt{2\pi}\cdot e^{-\pi^2/12}$. Finally,
\[
q^{-m}(1-q)^{-n}=\frac{(m+n)^{m+n}}{m^mn^n}\ge \binom{n+m}{n}\ge \binom{n+m-1}{n-1}.
\]
Combining the pieces, we obtain 
\begin{equation*}
I_{n,m} \ge {n+m-1 \choose n-1}  \exp\left( -\frac{\pi^2m}{6n}+O(\log n)\right). 
\end{equation*}
\end{proof}
The next Lemma provides an upper bound for $T_{n,m}$ applicable to all $m$, and
a lower bound for $T_{n,m}$ in the case when $m$ meets the condition of Lemma
\ref{lem: asym IS}, i.e., far beyond the constraint of Lemma \ref{ell components}.
\begin{Lemma}\label{bound,allmn} 
$(i)$ For all $m,n\ge 0$,
\begin{equation} \label{uppersimple}
T_{n,m}\le C_n I_{n,m} \le C_n\binom{n+m-1}{n-1}. 
\end{equation}
$(ii)$ If $n \to \infty$ and $m=o\bigl(n^{3/2}\bigr)$, then
\begin{equation}\label{lowersimple}
T_{n,m}\ge_b \, \exp\bigl[-2(m/n)(\log n)\bigr]C_n\,\binom{n+m-1}{n-1}.
\end{equation}
\end{Lemma}

Before we start our proof, we give a characterization of chord diagrams in terms of permutations of the multiset $\{1,1,2,2,\dots,n,n\}$. Consider a chord diagram $C$ with $n$ chords. If $x<y$ and $x$ and $y$ are the two endpoints of a chord in $C$, we denote this chord by $(x,y)$. Let $C=\{(p_i,q_i) : i\in [n] \}$, where $1\le p_1<\dots<p_n\le 2n-1$. Note that this representation of $C$ is unique and it gives a labeling of the chords of $C$; the $k$-th chord of $C$ is $(p_k,q_k)$ for $k \in [n]$. For all $k\in [n]$, we relabel the endpoints $p_k$ and $q_k$ with $k$, and also we color the point corresponding to $p_k$ with blue and the point corresponding to $q_k$ with red. At the end, we have $n$ blue points and $n$ red points and each set of points with the same color are labeled from 1 to $n$. In this relabeling and coloring, the blue $k$ represents the initial point and the red $k$ represents the terminal point of the $k$-th chord. Thus, a chord diagram determines a permutation of the union of a blue set $[n]$ and a red set $[n]$ such that the numbers in the blue set are in the natural order in the permutation and the red $k$ appears after the blue $k$ for each $k\in [n]$.

Conversely, any such permutation gives a unique chord diagram. To make it precise, 
let $\mathcal{S}_n^*$ be the set of permutations of the union of the two copies of $[n]$, one colored blue and the other red, such that 
\begin{enumerate}[itemsep=0pt, parsep=0pt, topsep=0pt, partopsep=0pt]
\item[(1)] the blue $i$ appears before the blue $(i+1)$ for $i \in [n-1]$ and
\item[(2)] the blue $j$ apears before the red $j$ for $j \in [n]$.
\end{enumerate}
Let $\pi$ be a permutation in $\mathcal{S}_n^*$. Let $p_i$ and $q_i$ be the positions of the blue and red $i$'s, respectively, in $\pi$, for $i \in [n]$. Then, the chord diagram corresponding to $\pi$ has the set of chords $(p_i,q_i)$ for $i\in [n]$.

Consider a permutation in $\Cal S_n^*$. Disregarding the labels but paying attention to the colors, we see a sequence of $n$ blue points and $n$ red points. Moreover, in this sequence, any prefix has at least as many blue points as red points since a particular  blue number appears before its red counterpart. Let $\C=\C(n)$ be the set of sequences of $n$ blue points and $n$ red points such that any prefix of a sequence in $\C$ has at least as many blue points as red points. We call the sequences in $\C$ \emph{Catalan sequences}. The cardinality of $\C$ is the $n$-th Catalan number $C_n$. 

A sequence $\boldsymbol s \in \C$ is determined uniquely by the numbers $y_1,\dots,y_n$, where $y_i$ is the number of red points between the $i$-th and $(1+i)$-th blue points for $i\in [n-1]$ and $y_n$ is the number of red points after the $n$-th blue point. Clearly, $y_1+\dots+y_n=n$ and $y_1+\dots+y_k\le k$ for $k<n$. We call the sequence $(y_1,\dots,y_n)$ \emph{the allocation sequence} of $\boldsymbol s$ with the interpretation that we allocate $n$ red dots to $n$ intervals determined by the consecutive blue points. By abuse of notation, we also call it the allocation sequence of $C$ if $\boldsymbol s$ is the corresponding sequence in $\C$ of $C$.

Finally, we introduce a sequence $\boldsymbol x=(x_1,\dots,x_n)$ for a chord diagram $C$, called the \emph{intersection sequence} of $C$, which is analogous to the inversion sequence of a permutation. For a chord diagram $C$, consider the labeling of the chords described above. Then, $x_j$ is defined to be the number of chords crossing the $j$-th chord whose labels are smaller than $j$, that is, 
\begin{equation} \label{defn, intersection sequence}
x_j=x_j(C)=|\{i<j: \text{ chord } i \text{ crosses chord } j \text{ in } C \}|.
\end{equation}
Note that $0\le x_j\le j-1$ and the number of crossings in $C$ is equal to $x_1+\dots+x_n$. An intersection sequence is an inversion sequence but there might be many chord diagrams with the same intersection sequence.

For a chord diagram $C$, let $\pi(C)$ denote the permutation in $\Cal S_n^*$ corresponding to $C$. For $i<j$, chord $i$ and chord $j$ cross each other in $C$ if,  in $\pi(C)$, first the blue $i$ appears, then the blue $j$, then the red $i$, and finally the red $j$. 

\begin{proof}[Proof of Lemma \ref{bound,allmn}]
$(i)$ Let $\Cal S_n^*$ and $\C$ be as defined above. Let $\Cal T=\Cal T_{n,m}$ be the set of chord diagrams of size $n$ with $m$ crossings and let $\I=\I(n,m)$ be the set of inversion sequences $\boldsymbol x=(x_1,\dots,x_n)$ such that $x_1+\cdots+x_n=m$.
By the previous discussion, any chord diagram determines a unique permutation in $\Cal S_n^*$, a unique sequence in $\C$, and a unique sequence in $\I$. Therefore, for any $C\in \mathcal T$, there is at most one $(\bs s, \bs x) \in \C\times \I$, which implies
\[
T_{n,m} \le C_n I_{n,m}.
\]
The second inequality of part $(i)$ follows from \eqref{upper I_n,m}. 

$(ii)$ Take an arbitrary permutation $\bs s$ in $\C$ and an arbitrary sequence $\boldsymbol x$ in $\I$. Let $\boldsymbol y$ be the allocation sequence of $\boldsymbol s$. Put the blue numbers $1$ to $n$ on a line in the same order from left to right and reserve $y_i$ spaces between the blue $i$ and the blue $i+1$. If 
there is a chord diagram $C$ whose allocation and intersection sequences are $\boldsymbol y$ and $ \boldsymbol x$, respectively,  we must have
\begin{equation}\label{xandy}
x_{n-k} \le y_{n-k}+y_{n-k+1}+\dots+y_{n}- (k+1), \quad 0\le k\le n-1
\end{equation}
for the following reason. In the chord diagram $C$, the chord labeled with $(n-k)$ intersects $x_{n-k}$ chords of smaller labels and thus there must be at least $x_n$ red numbers smaller than $(n-k)$ appearing after blue $(n-k)$ in the permutation representation of $C$. On the other hand, there are $y_{n-k}+\cdots+y_n$ red numbers appearing after blue $(n-k)$, of which $k+1$ of them are $n-k,\dots,n$. Thus, the number of smaller red numbers after blue $(n-k)$, which is $y_{n-k}+\cdots+y_n-(k+1)$, must be at least as large as $x_{n-k}$.

Conversely, if \eqref{xandy} is satisfied, then there is a chord diagram $C$ with the allocation and intersection sequences $\boldsymbol y$ and $ \boldsymbol x$, respectively, which can be determined by placing the red numbers to the appropriate spots starting from $n$ and proceeding backwards. First, red $n$ is placed to the $(1+x_n)$-th available space after the blue $n$. Once red $n,(n-1),\dots,(n-k+1)$ are placed, to guarantee that the $(n-k)$-th chord intersects $x_{n-k}$ chords of smaller label, red $(n-k)$ is placed to the $(1+x_{n-k})$-th unoccupied spot (from left to right) to the right of blue $(n-k)$. 

For a given $\bs x$, let $N(\bs x)$ denote the number of $\bs y$'s meeting the constraint \eqref{xandy}.  By Lemma~\ref{max x_i, M} below, $N(\bs x)$ is at least $C_{n-M}$, where $M=M(\bs x)$ denotes the maximum of $x_i$'s in $\bs x$. Consequently,
\beq \label{eq: lower T_n,m}
T_{n,m}=\sum_{\bs x \in \I} N(\bs x) \ge \sum_{\bs x \in \I} C_{n-M(\bs x)}.
\eeq

In the proof of Lemma 3.4 in \cite{AP} it was shown that, whp, the maximum $M(\bs x)$ does not exceed 
$(1+\eps)(m/n)\log n$ when the sequence $\boldsymbol x$ is chosen uniformly at random from $\I$. Using this fact and \eqref{eq: lower T_n,m}, we get
\begin{equation}\label{lowerbound}
I_{n,m}C_{n-M_0} \lesssim  T_{n,m}
\end{equation}
for $M_0=\lceil ((1+\eps)m/n)\log n \rceil$. 
Also, by  the Stirling's formula for the factorials, 
\begin{equation}\label{Cn-M0sim}
C_{n-M_0} \sim 4^{-M_0}\cdot C_n = \exp(-\log(4)M_0)C_n.
\end{equation}
Combining \eqref{lowerbound} and
\eqref{Cn-M0sim}, with small enough $\eps$,  and Lemma~\ref{lem: asym IS} we complete the proof, pending the proof of the next Lemma \ref{max x_i, M}.
\end{proof}

\begin{Remark}
A closer look shows that, in fact, $M(\bs x)$ is asymptotic to $(m/n)\log n$ in probability,
and that $\pr(M(\bs x)\le (1-\eps)(m/n)\log n)\le \exp(-c n^{\eps})$, which is much smaller than
$\exp(-\Theta(m/n))$.  Thus, the choice of $M_0$ in \eqref{lowerbound} is asymptotically
the best possible if we want the fraction $I_{n,m,M_0}/I_{n,m}$ to be at least $e^{-bm/n}$ for some constant $b>0$; here the $I_{n,m,M_0}$ denotes the number of permutations with $m$ inversions and $\max x_i \le M_0$.
\end{Remark}

\begin{Lemma}\label{max x_i, M}
 Let $\boldsymbol x=(x_1,\dots,x_n)$ be a given sequence in $\mathcal{X}_n(m)$, i.e., an inversion sequence with $m$ total inversions. Let $M=M(\boldsymbol x)$ denote the maximum of the $n$ terms in this sequence. Then, there are at least $C_{n-M}$ sequences $\boldsymbol y$ satisfying \eqref{xandy}. Therefore, $N(\boldsymbol x) \ge C_{n-M}$.
 \end{Lemma}

\begin{proof}
We say that a sequence $\boldsymbol s$ of blue and red points is a \textit{Catalan sequence} if $\boldsymbol s$ has the same number of blue and red points and any prefix of $\boldsymbol s$ has at least as many blue points as red points. Then, the set $\C$ is the set of Catalan sequences of length $2n$. 

Let $\Cal A=\Cal A(M)$ be the set of sequences of $n$ blue points and $n$ red points such that any sequence in $\Cal A$ starts with $M$ blue points and ends with $M$ red points, and the subsequence consisting of $2n-2M$ remaining points is a Catalan sequence. Clearly, a sequence in $\Cal A$ is a Catalan sequence itself, since prepending $M$ blue points and appending $M$ red points to a Catalan sequence results in a Catalan sequence. Thus, $\Cal A$ is a subset of $\C$ and the size of $\Cal A$ is $C_{n-M}$. Let $\boldsymbol s \in \Cal A$ and let $\boldsymbol y=(y_1,\dots,y_n)$ denote the allocation sequence of $\boldsymbol s$. Let $\boldsymbol y'=(y'_1,\dots, y'_{n-M})$ be the allocation sequence of the sequence obtained from $\boldsymbol s$ by removing the first and the last $M$ elements. Then, we have 
\[ y_1=\cdots=y_M=0, \quad y_{M+i}=y'_i \quad \text{for } 1\le i\le n-M-1, \quad y_n=y'_{n-M}+M. 	\]
Since $y'_{n-M}+\cdots+y'_{n-M-k}-(k+1)\ge 0$ in the sequence $\boldsymbol y'$ for all $0\le k\le n-M-1$, we have
\[	y_{n}+\cdots+y_{n-k}-(k+1)\ge M			\]
for all $0\le k\le n-M-1$. On the other hand, for $j\le M$, we have
\[
y_j+\cdots+y_n-(n-j-1) = y_{M+1}+\cdots+y_{n}-(n-j+1) =j-1 \ge x_j
\]
since $\boldsymbol x$ is an inversion sequence.
Thus, \eqref{xandy} holds for the sequences $\boldsymbol x$ and $\boldsymbol y$. Consequently, we have 
$C_{n-M} \le N(\boldsymbol x)$.
\end{proof}

\begin{Remark} Our, admittedly limited, numerical experiments seem to indicate that,
for $m=\Theta(n\log n)$, $T_{n,m}$ is at least of order $e^{-b(m/n)} C_n\binom{n+m-1}{n-1}$ for some constant $b>0$,
a bound that matches qualitatively the asymptotic formula for $T_{n,m}$ for $m< (2/\pi^2)n\log n$
in Lemma \ref{5}. However, the exponential factor in the lower bound \eqref{lowersimple}
is much smaller, namely $e^{-\Theta((m/n)^2)}$. So far we have not been able to replace this
factor by anything substantially larger. 
At the moment, it seems that $n^{3/2}$ is actually the threshold value of m
for validity of the lower bound. Here is a quick-and-dirty argument to lend some support for
this conjecture. An intersection sequence $\bs x$ and an allocation sequence $\bs y$ determine
 a chord diagram if and only if the condition  \eqref{xandy} is satisfied. Further, $\bs y$ corresponds to a Catalan path $L$ from $(0,0)$ to $(n,n)$ in the integer lattice with right and up moves not crossing the diagonal. Let $L_i:=\min \{j: (i,j)\in L\}$. The right side of \eqref{xandy} is the same as $(n-k)-L_{n-k}$, the vertical distance between the diagonal and the lowest point of the path after $n-k$ right moves. In a typical Catalan path, $\max_i (i-L_i)$ is of order $O(\sqrt{n})$. For $m/n\gg n^{1/2}$, average $x_i$ is $m/n$, which is much greater than the maximum of $i-L_i$. As a result, the probability that $x_i\le i-L_i$ for all $i\in [n]$ for a random $\bs x$ and a random $\bs y$ is extremely small. We do not know though how to handle non-typical Catalan paths;  so we cannot exclude the possibility that the conjecture is false.
\end{Remark}

\begin{Remark} By Lemma \ref{bound,allmn}, \eqref{uppersimple},  for $x,y>0$,
\begin{align*}
T(x,y)=&\,\sum_{m,n}T_{n,m} x^my^n\le \sum_{m,n} \binom{n+m-1}{n-1}C_n x^m y^n\\
=&\sum_n y^n C_n \sum_m x^m \binom{n+m-1}{n-1}.
\end{align*}
For $x<1$, the innermost series converges to $(1-x)^{-n}$, and then the double series
converges to $C(y/(1-x))$ if $y/(1-x)<1/4$. Therefore, we have an elementary proof that
the bivariate generating function series $T(x,y)$ converges if $x,y>0$ and $y/(1-x)<1/4$. 
\end{Remark}

Here is an illustration of the power of the upper bound \eqref{uppersimple} 
combined with Lemma~\ref{4}. Consider again  the uniformly random linearized chord diagram on $[2n]$ with $m$ crossings. For a cut $\Cal C$ with the partition $[2n_1]\cup [2n_1+1,\dots 2n]$,  we set $n_2=n-n_1$, define $|\Cal C|=\min\{n_1,n_2\}$, and finally define $Y_{n,m}=\max_{\Cal C} |\Cal C|$.

\begin{Lemma}\label{5} If $m=O(n)$ then $Y_{n,m}$ is bounded in probability.
\end{Lemma}
\begin{proof} Given $n_1+n_2=n$ and $m_1+m_2=m$, where $n_1,n_2>0$, the expected number
of cuts with parts $[2n_1]$ and $[2n_1+1,2n]$, and the number of crossings in the left subdiagram and the right 
subdiagram equal $m_1$ and  $m_2$, respectively, is 
\[
Z_{\bs n,\bs m}:=\frac{T_{n_1,m_1}T_{n_2,m_2}}{T_{n,m}}.
\]
Here,
by Lemma \ref{4}, 
\[
T_{n,m}\sim\binom{n+m-1}{n-1}C_n\prod_j(1-q^j)^3,
\]
and, by \eqref{uppersimple},
\[
T_{n_i,m_i}\le \binom{n_i+m_i-1}{n_i-1}C_{n_i},\quad i=1,2.
\]
Hence, 
\[
Z_{\bs n,\bs m}\le_b\frac{\prod_i\binom{n_i+m_i-1}{n_i-1}C_{n_i}}{\binom{n+m-1}{n-1}C_n}
\le \frac{\prod_i\binom{n_i+m_i}{n_i}C_{n_i}}{\binom{n+m}{n}C_n}.
\]
Therefore, since $C_{\nu}=\Theta\bigl(\nu^{-3/2}4^{\nu}\bigr)$,
\begin{equation}\label{nu3/2}
Z_{\bs n,\bs m}\le_b\frac{n^{3/2}}{n_1^{3/2}n_2^{3/2}}\cdot\frac{\prod_i\binom{n_i+m_i-1}{n_i-1}}{\binom{n+m-1}{n-1}}.
\end{equation}
Observe that
\[
\sum_{\bs m: m_1+m_2=m}\prod_i\binom{n_i+m_i-1}{n_i-1}=\binom{n+m-1}{n-1}.
\]
Indeed, the RHS is the total number of non-negative integer solutions of
\[
\sum_{j=}^{n_1} x_j  + \sum_{j=n_1+1}^{n} x_j = m,
\]
and each such solution is a pair $(x_1,\dots,x_{n_1})$, $(x_{n_1+1},\dots,x_{n_1+n_2})$
of solutions, each of the corresponding equation
\[
\sum_{j=}^{n_1} x_j = m_1,\quad \sum_{j=n_1+1}^{n} x_j  = m_2,
\]
for the unique choice of $m_1$, $m_2$ satisfying $m_1+m_2=m$. So summing \eqref{nu3/2}
over $\bs m$, we get 
\[
\sum_{\bs m: m_1+m_2=m}Z_{\bs n,\bs m}\le_b\frac{n^{3/2}}{n_1^{3/2}n_2^{3/2}}.
\]
Consequently, as $A\to \infty$,
\begin{align*}
\pr(Y_{n,m}\ge A) &\le\sum_{\bs n:\min\{n_1,n_2\}\ge A}\,\,\sum_{\bs m:m_1+m_2=m}Z_{\bs n, \bs m}\\
&\le_b n^{3/2}\sum_{\min\{n_1,n_2\}\ge A} n_1^{-3/2}n_2^{-3/2}\\
&\le_b \sum_{A\le n_1\le n/2}n_1^{-3/2}=O(A^{-1/2})\to 0. \qedhere
\end{align*}
\end{proof}

\section{The Largest Component}\label{sec: large components}

We now turn our attention to the component sizes of chord diagrams with given number $m$ of
crossings. Throughout this section, unless otherwise stipulated, we will assume that $m$
satisfies the condition \eqref{m leq} with $\ell=1$, so that $T_{n,m}$ is given by the asymptotic formula \eqref{asym,ext,ell} with $\ell=1$.

We need a usable bound for $C_{\nu,\mu}$, the total number
of connected chord diagrams on $[2\nu]$ with $\mu$ crossings. 

It was first found by Dulucq and Peanud \cite{DP} 
(see also Stanley~\cite[Exercise 5.46]{St}) that  $C(\nu,\nu-1)=\frac{1}{2\nu-1}{3\nu-3 \choose \nu -1}$, and Acan \cite{Acan} proved that
\begin{align*}
C(\nu,\nu)&= 2+\sum_{j=1}^{6\wedge(\nu-3)}\frac{\nu}{3}\binom{6}{j}\frac{j}{\nu-3}\binom{3\nu-9}
{\nu-3-j}\\
&\quad+2\sum_{k=4}^{\nu-1}\frac{\nu}{k}\sum_{j=1}^{(\nu-k)\wedge2k}\frac{j}{\nu-k}\binom{2k}{j}
\binom{3\nu-3k}{\nu-k-j}.
\end{align*}
Thus, $C(\nu,\nu-1)=\Theta(\nu^{-1}\binom{3\nu}{\nu}$, $C(\nu,\nu)=\Theta(\binom{3\nu}{\nu})$,
and we conjecture that for $\mu-\nu=O(\nu^{\eps})$, $\eps>0$ being small,
$C(\nu,\mu)=\Theta(\nu^{\mu-\nu}\binom{3\nu}{\nu})$.

A chord diagram counted in $T_{n,m}$ contains a component including point $1$,
with $2\nu$ points ($\nu$ chords) and $\mu$ crossings, with remaining $2n-2\nu$ points forming $2\nu$
subintervals, of sizes $2n_1,2n_2,\dots,2n_{2\nu}$, ordered clockwise starting after point $1$. Observe
that there can be no chords containing two points from two different subintervals, since any
such chord would have crossed at least one of the $\nu$ chords. Therefore, we will have $2\nu$ isolated subdiagrams, with crossing numbers
$m_1,\dots,m_{2\nu}$, adding up to $m-\mu$. Thus,
\[
T_{n,m}=\sum_{\nu\ge 1,\mu\ge \nu-1}C_{\nu,\mu}\sum_{n_1+\cdots+n_{2\nu}=n-\nu
\atop m_1+\cdots+m_{2\nu}=m-\mu}\prod_{j=1}^{2\nu}T_{n_j,m_j}.
\]
Setting $C_{0,0}=1$, we get
\begin{align}\label{T=C}
\sum_{n,m} T_{n,m}x^my^n  &=  1+\sum_{\nu,\mu}C_{\nu,\mu}x^{\mu}y^{\nu}\sum_{n_1+\dots+n_{2\nu}>0\atop m_1,\dots,m_{2\nu}\ge 0}\prod_{j=1}^{2\nu}T_{n_j,m_j} x^{m_j}y^{n_j}  \notag 	\\
&=\sum_{\nu,\mu}C_{\nu,\mu}x^{\mu}y^{\nu} \left(\sum_{n_1\ge 0,m_1\ge 0}T_{n_1,m_1}x^{m_1}y^{n_1}\right)^{2\nu}.
\end{align}
Equivalently,
\begin{equation}\label{"expon"}
T(x,y)=C(x,yT^2(x,y)),
\end{equation}
where $C(x,y):=\sum_{\mu,\nu}C_{\nu,\mu}x^{\mu}y^{\nu}$ denotes the bivariate generating 
function for the sequence $\{C_{\nu,\mu}\}$.
Equation \eqref{"expon"} implies  a Chernoff-type bound for $C_{\nu,\mu}$:
\begin{equation}\label{Cnumu<,Cher}
C_{\nu,\mu}\le \frac{T(x,y)}{x^{\mu}y^{\nu} [T(x,y)]^{2\nu}},\quad\forall\, x<1,\,\,y<\frac{1-x}{4}.
\end{equation}
Since $T(x,y)$ increases with $y$, the best estimate, for a given $x<1$, is obtained
by letting $y\uparrow (1-x)/4$.  From \eqref{Tou}, and $C(1/4)=2$, it follows that 
\[
\lim_{y\uparrow (1-x)/4}T(x,y)=2f(x),\quad f(x)=\sum_{j\ge 0}(-1)^j x^{\binom{j+1}{2}}.
\]
Therefore, the bound \eqref{Cnumu<,Cher} becomes
\begin{equation}\label{x<1}
C_{\nu,\mu}\le \frac{2f(x)}{x^{\mu}(1-x)^{\nu}f(x)^{2\nu}}\le  \frac{2}{x^{\mu}(1-x)^{\nu}f(x)^{2\nu}} ,\quad\forall\, x<1.
\end{equation}
Using $f(x)>1-x$ in \eqref{x<1}, we obtain
\[
C_{\nu,\mu}\le \frac{2}{x^{\mu}(1-x)^{3\nu}},\quad\forall\,x<1.
\]
The RHS is minimized at $x=\mu/(3\nu+\mu)$, and we get
\begin{equation}\label{nu+3mu}
C_{\nu,\mu}\le 2\, \frac{(3\nu+\mu)^{3\nu+\mu}}{(3\nu)^{3\nu}\mu^{\mu}}.
\end{equation}
In particular,
\begin{equation}\label{trees<}
C_{\nu,\nu-1}\le_b \frac{(4\nu-1)^{4\nu-1}}{(3\nu)^{3\nu}(\nu-1)^{\nu-1}}\le_b \nu^{1/2}\binom{4\nu}{\nu},
\end{equation}
similar to, but noticeably worse than the exact formula for the number of trees, which is $\frac{1}{2\nu-1}{3\nu-3 \choose \nu-1}$. For $\mu/\nu$
large, we get a bound better than \eqref{nu+3mu} by using the obvious $C_{\nu,\mu}<T_{\nu,\mu}$
and Equation~\eqref{uppersimple}:
\begin{equation}\label{Cnumu,mu/nularge}
C_{\nu,\mu}\le C_{\nu}\binom{\mu+\nu-1}{\nu-1}\le_b \frac{4^{\nu}}{\nu^2}\cdot\frac{(\mu+\nu)^{\mu+\nu}}
{\mu^{\mu}\nu^{\nu}} \le 4^{\nu}\frac{(\mu+\nu)^{\mu+\nu}}
{\mu^{\mu}\nu^{\nu}}.
\end{equation}
Combining \eqref{nu+3mu} and \eqref{Cnumu,mu/nularge}, we obtain
\beq \label{upper Cnumu}
C_{\nu,\mu} \le_b \, \min \bigg\{  \frac{4^{\nu}}{\nu^2}\cdot\frac{(\mu+\nu)^{\mu+\nu}}
{\mu^{\mu}\nu^{\nu}},    \frac{(3\nu+\mu)^{3\nu+\mu}}{(3\nu)^{3\nu}\mu^{\mu}}	\bigg\}.
\eeq

\begin{Lemma}[Crossing-density gap]\label{densitygap} 
Let $\alpha$ be a constant greater than $4e^2$ and let $\beta=\frac{5}{\log\alpha-\log(4e^2)}$. For crossing density $m/n$ exceeding $\alpha$, whp, there is no component of size above $\beta\log n$  whose edge density is below $ m/(\alpha n)$.
\end{Lemma}

\begin{proof}  Let us first  bound the expected number of components with
parameters $(\nu,\mu)$ in the random circular diagram on $[2n]$ with $m$ crossings.
Any such component that contains vertex $1$  induces a partition  of $[2n]$ minus a subset of $2\nu$ points
into $2\nu$ clockwise ordered subintervals
with $2n_1,\cdots, 2n_{2\nu}$ points respectively (possibly with some nonpositive $n_i$'s),
corresponding to $2\nu$ {\it isolated\/} subdiagrams
with $m_1,\cdots, m_{2\nu}$ crossings respectively, with $m_1 + ... +m_{2\nu} = m - \mu$.
So the probability $P_{\nu,\mu}$ that vertex $1$ belongs to a component with parameters $(\nu,\mu)$ 
is given by 
\begin{equation*}
P_{\nu,\mu}=\frac{C_{\nu,\mu}}{T_{n,m}}\sum\limits_{n_1+\cdots+n_{2\nu}=n-\nu\atop
m_1+\cdots+m_{2\nu}=m-\mu}\prod_{j=1}^{2\nu}T_{n_j,m_j}=\frac{C_{\nu,\mu}}{T_{n,m}}\,
[x^{m-\mu}y^{n-\nu}]\, T(x,y)^{2\nu}.
\end{equation*}
 Let $X_{\nu,\mu}$ be the total number of
components with parameters $\nu$ and $\mu$. Then, by cyclic symmetry, 
\begin{equation}\label{expcomp}
\ex[X_{\nu,\mu}]=P_{\nu,\mu}\,\frac{n}{\nu}
=\frac{n\,C_{\nu,\mu}}{\nu\, T_{n,m}}\,[x^{m-\mu}y^{n-\nu}]\, T(x,y)^{2\nu}.
\end{equation}
Let us see what we can get from \eqref{expcomp}. By \eqref{Tou},
\begin{equation*}
[x^{m-\mu}y^{n-\nu}]\, T(x,y)^{2\nu}
\le \frac{\bigl[C(y/(1-x)) A(x,1-C(y/(1-x)))\bigr]^{2\nu}}{x^{m-\mu}y^{n-\nu}},
\end{equation*}
for all $x<1$, $y<(1-x)/4$. Letting $y\uparrow (1-x)/4$, using $C(1/4)=2$ and $A(x,-1)=f(x)\le 1$, and setting 
$x=(m-\mu)/(m-\mu+n-\nu)$, we obtain
\begin{align*}
[x^{m-\mu}y^{n-\nu}]\, T(x,y)^{2\nu}\le&\, 4^n\,\frac{1}{x^{m-\mu}(1-x)^{n-\nu}}\\
=&\,4^n\, \frac{(m-\mu+n-\nu)^{m-\mu+n-\nu}}{(m-\mu)^{m-\mu}(n-\nu)^{n-\nu}}.
\end{align*}
Consequently, the identity \eqref{expcomp} yields 
\begin{equation}\label{EXnumu<}
\ex[X_{\nu,\mu}]\le nC_{\nu,\mu}\cdot \frac{4^n}{T_{n,m}}\,\frac{(m-\mu+n-\nu)^{m-\mu+n-\nu}}{(m-\mu)^{m-\mu}(n-\nu)^{n-\nu}}.
\end{equation}
Now, by Lemma \ref{ell components}, 
\begin{equation*}
T_{n,m}\ge_b \binom{n+m-1}{n-1}C_n \exp\left(-\frac{\pi^2}{2(1-q)}\right),\quad
q:=\frac{m}{m+n},
\end{equation*}
provided that 
\begin{equation}\label{m<}
m\le \frac{2n}{\pi^2}\bigl(\log n -2\log\log n -\omega(n)\bigr), 
\end{equation}
where $\omega(n)\to\infty$  however slowly. For this $m$, we have
\[
\exp\left( \tfrac{\pi^2}{2(1-q)}\right)\le \frac{n}{\log n} \quad \text{and} \quad {n+m-1 \choose n-1} \ge \frac{{n+m \choose n}}{ \log n}.
\]
Using the two inequalities above and Stirling's formula for the Catalan number $C_n$, Equation~\eqref{EXnumu<} becomes
\begin{equation*}
\ex[X_{\nu,\mu}]\le_b  \frac{n^{7/2}\cdot C_{\nu,\mu}}{{n+m \choose n}}\cdot \frac{(m-\mu+n-\nu)^{m-\mu+n-\nu}}{(m-\mu)^{m-\mu}(n-\nu)^{n-\nu}}.
\end{equation*}
Now, using 
\[
\text{const }b^{-1/2}\frac{b^b}{a^a(b-a)^{b-a}}\le \binom{b}{a}\le \frac{b^b}{a^a(b-a)^{b-a}},
\]
and log-concavity of $f(a,b):=\tfrac{b^b}{a^a(b-a)^{b-a}}$, we replace the last bound with
a cruder version. Namely, if $m$ satisfies \eqref{m<}, then
\begin{equation}\label{mean Xnumu}
\ex[X_{\nu,\mu}]\le_b n^4 C_{\nu,\mu}\,\frac{n^{\nu} m^{\mu}}{(n+m)^{\nu+\mu}}
\end{equation}
uniformly for all $\nu\le n$ and $\nu-1\le\mu\le m$, or using \eqref{Cnumu,mu/nularge},
\begin{equation}\label{EXnumu,cruder}
\begin{aligned}
\ex[X_{\nu,\mu}]\le_b&\, n^4\,\frac{4^{\nu}}{\nu^2}\,\frac{(\mu+\nu)^{\mu+\nu}}
{\mu^{\mu}\nu^{\nu}}\,\frac{n^{\nu} m^{\mu}}{(n+m)^{\nu+\mu}}\\
=&\,n^4\, (4^{\nu}/\nu^2) F(\mu/\nu, m/n)^{\nu},\\
F(x,y):=&\,\frac{(1+x)^{1+x}}{x^x}\cdot\frac{y^x}{(1+y)^{1+x}}.
\end{aligned}
\end{equation}
Let  $y=m/n\ge \alpha$ with $\alpha>4e^2$, and $x=\mu/\nu \le y/\alpha$. Taylor-expanding $z\log z$ about
$z=x$ and using $x\ge (\nu-1)/\nu$,
\begin{align*}
\log F(x,y)=&\,(1+x)\log(1+x)-x\log x +x\log y - (1+x)\log(1+y)\\
\le&\, (1+\log x) +\frac{1}{2x} -\log y\le 2+\log x-\log y=\log\frac{xe^2}{y}\\
\le&\, \log\frac{ye^2}{\alpha y}=\log\frac{e^2}{\alpha}<\log\frac{1}{4}.
\end{align*}
So \eqref{EXnumu,cruder} becomes:
\begin{equation*}
\ex[X_{\nu,\mu}]\le_b n^4 \rho^{\nu}/\nu^2,\quad \rho:=\frac{4 e^2}{\alpha},
\end{equation*}
uniformly for all $\mu\ge \nu-1$ with $\mu/\nu\le \alpha^{-1} m/n$. For $\beta=-5/\log \rho$,
\begin{align*}
\sum_{\nu\ge \beta\log n,\, \mu/\nu\le \alpha^{-1}m/n}\ex[X_{\nu,\mu}]\le_b&\,
n^4(m/n) \sum_{\nu\ge \beta \log n}  \rho^{\nu}/\nu \\
\le_b&\, n^4 \log n\cdot \rho^{\beta \log n}/\log n=1/n \to 0. \qedhere
\end{align*}
\end{proof}

Lemma \ref{densitygap} shows that, for the random diagram with density $m/n$ sufficiently large, whp  there are no components of size $\Omega(\log n)$ with density smaller by a constant
factor than
$m/n$. We anticipate that, for $m/n\to\infty$, whp there exists a large component and that
a likely candidate is a component with the maximum density. Let us focus on such 
components. Given parameters $\nu$ and $\mu$, let $A_{\nu,\mu}$ denote the
event  ``there is a maximum density component with $\nu$ chords and $\mu$ crossings''.
Needless to say, on the event $A_{\nu,\mu}$, the maximum density is $\mu/\nu$.

\begin{Lemma}\label{maxden} Suppose $m/n\to\infty$ and $m$ satisfies \eqref{m leq} in
Lemma \ref{ell components}.  Let $c\in (1,2)$ be fixed. Define $\alpha =7\max\{\log(1/ce^{-c}),\,\log(1/0.99)\}$.
Then, 
\begin{equation}\label{sumPAnumu}
\lim_{n,m\to\infty}\sum_{\nu,\mu}\pr(A_{\nu,\mu})=0,
\end{equation}
where the sum is over all pairs $(\nu,\mu)$ such that 
\[
\alpha \log n\le\nu,\quad \mu\le (2-c) m.
\]
In words, it is very unlikely that the densest component has size exceeding $\alpha \log n$
and that its number of crossings scaled by $m$ is strictly below $1$.
\end{Lemma}

\begin{proof} Notice upfront that $\pr(A_{\nu,\mu})=0$ if $\mu/\nu<m/n$. Thus, in \eqref{sumPAnumu},
the terms of interest are those with $\mu/\nu\ge m/n$.
 As in the proof of Lemma \ref{densitygap},  a component with parameters $\nu$ and $\mu$ induces the partition of the remaining set of $2(n-\nu)$ points into $2\nu$ isolated
subdiagrams with parameters $n_j,m_j$, $1\le j\le 2\nu$.  If a chosen component is
of maximum density $\mu/\nu$, then, in addition, we must have $m_j/n_j\le \mu/\nu$. So,
instead of \eqref{expcomp}, we obtain
\begin{equation*}
\pr(A_{\nu,\mu})\le 
\frac{n\,C_{\nu,\mu}}{\nu\, T_{n,m}}\,[x^{m-\mu}y^{n-\nu}]\, T_{\mu/\nu}(x,y)^{2\nu},
\end{equation*}
where
\[
T_{\mu/\nu}(x,y):=1+\sum_{0<i/j\le\mu/\nu} T_{i,j}x^iy^j.
\]
Here
\[
[x^{m-\mu}y^{n-\nu}]\, T_{\mu/\nu}(x,y)^{2\nu}\le \frac{T_{\mu/\nu}(x,y)^{2\nu}}{x^{m-\mu}y^{n-\nu}},
\quad\forall\, x>0,\,y>0.
\]
Let 
\[
x:=\frac{m-\mu}{m-\mu+n-\nu},\quad y:=\frac{1}{4} (1-x)
\]
and observe that $x\to 1$ from below since $m-\mu\ge m(c-1)\gg
n$. Similar to \eqref{mean Xnumu}, we obtain
\begin{align}
\pr(A_{\nu,\mu})\le&\,
\frac{n\,C_{\nu,\mu}4^{n-\nu}}{\nu\, T_{n,m}}\,\frac{(m-\mu+n-\nu)^{m-\mu+n-\nu}}
{(m-\mu)^{m-\mu}(n-\nu)^{n-\nu}}\,T_{\mu/\nu}(x,y)^{2\nu}\notag\\
\le_b&\,n^4 4^{-\nu}C_{\nu,\mu}\,\frac{m^{\mu}n^{\nu}}{(m+n)^{\mu+\nu}}\,T_{\mu/\nu}(x,y)^{2\nu}
\notag\\
\le&\,\frac{n^4}{\nu^2}\cdot\frac{(\mu+\nu)^{\mu+\nu}}{\mu^{\mu}\nu^{\nu}}\cdot\frac{m^{\mu}n^{\nu}}{(m+n)^{\mu+\nu}}\,T_{\mu/\nu}(x,y)^{2\nu}, \label{PAnumu<}
\end{align}
where we use \eqref{Cnumu,mu/nularge} in the last step. Let us bound the last factor in \eqref{PAnumu<}. Using the upper bound \eqref{uppersimple} in
Lemma \ref{bound,allmn}, we have
\begin{align}\label{T< Sigma1-Sigma2}
T_{\mu/\nu}(x,y)  & \le 1+ \sum_{0< i/j\le \mu/\nu}\binom{i+j-1}{j-1}C_jx^iy^j	\notag	\\
&=\sum_{i,j\ge 0}\binom{i+j-1}{j-1}C_jx^iy^j -\sum_{j>0\atop i/j>\mu/\nu}\binom{i+j-1}{j-1}C_jx^iy^j \notag \\
&=:\Sigma_1 - \Sigma_2.
\end{align}
Here
\begin{align}
\Sigma_1=&\,\sum_{j\ge 0}C_j y^j\sum_{i\ge 0}\binom{i+j-1}{j-1}x^i\notag\\
=&\,\sum_{j\ge 0}C_j y^j (1-x)^{-j}=\sum_{j\ge 0}C_j (1/4)^j =C(1/4)=2.\label{Sigma1=}
\end{align}
Turn to $\Sigma_2$. For a given $j>0$, introduce $i_0=i_0(j):=\min\{i: i>j\mu/\nu\}$, and write
\[
\sum_{i>j\mu/\nu} \binom{i+j-1}{j-1} x^i=x^{i_0}\sum_{i\ge i_0}\binom{i+j-1}{j-1}x^{i-i_0}:=
x^{i_0}\Sigma_2^*.
\]
We are going to use Abelian summation by parts to bound $\Sigma_2^*$ from below. Using
\[
\sum_{b=a}^{a+N-1}\binom{b}{a}=\binom{a+N}{a+1},
\]
we have: for $N>0$,
\begin{align*}
S_{N,j}:=\sum_{i=i_0}^{i_0+N-1}\binom{i+j-1}{j-1}=&\,\sum_{i=0}^{i_0+N-1}\binom{i+j-1}{j-1}
-\sum_{i=0}^{i_0-1}\binom{i+j-1}{j-1}\\
=&\,\binom{i_0+j+N-1}{j}-\binom{i_0+j-1}{j},
\end{align*}
and $S_{0,j}=0$. Using 
\[
\binom{i+j-1}{j-1}=S_{i-i_0+1,j}-S_{i-i_0,j},
\]
we compute
\begin{align}
\Sigma_2^*=&\,\sum_{i\ge i_0}\bigl[S_{i-i_0+1,j}-S_{i-i_0,j}\bigr]x^{i-i_0}
=(1-x)\sum_{i\ge i_0}S_{i-i_0+1,j}x^{i-i_0}\notag\\
=&\,(1-x)\sum_{i\ge i_0}\left[\binom{i+j}{j}-\binom{i_0+j-1}{j}
\right] x^{i-i_0}\notag\\
\ge&\, (1-x)\sum_{i\ge i_0}\binom{i-i_0+j}{j}x^{i-i_0}\notag\\
=&\,(1-x)\cdot(1-x)^{-j-1}=(1-x)^{-j}.\label{Sigma2*>}
\end{align}
Explanation for  the inequality: First, 
\[
{i_0+j-1 \choose j} \le {i_0+j \choose j}-1; \quad j\ge 1,
\]
and then
\[
{i_0+j \choose j} + {i-i_0+j \choose j}-1 \le {i+j \choose j}
\]
since the number of ways to select $j$ balls from $i$ red balls and
$j$ white balls is at least the number of ways to choose $j$ balls from a subset of $i-i_0$ white balls and $j$ white balls plus the number of selections in which red balls, if any, have to
be selected from the complementary subset of $i_0$ balls. We subtract 1 from the left side since choosing $j$ white balls is counted twice. As $x\to 1$, we have $x^{i_0}=x^{j\mu/\nu}(1+O(1-x))$. Using \eqref{Catalan solution} and \eqref{Sigma2*>}, we have
\begin{align}
\Sigma_2=&\,\sum_{j\ge 0}C_jy^j x^{i_0(j)}\Sigma_2^*
\ge(1+O(1-x)))\sum_{j\ge 0}C_jy^j \left(\frac{x^{\mu/\nu}}{1-x}\right)^j\notag\\
=&\,(1+O(1-x))\sum_{j\ge 0}C_j (x^{\mu/\nu}/4)^j
=(1+O(1-x))C(x^{\mu/\nu}/4)\notag\\
=&\,(1+O(1-x))\frac{2}{1+\sqrt{1-x^{\mu/\nu}}}.\label{Sigma2=}
\end{align}
Combining \eqref{Sigma1=} and \eqref{Sigma2=} we transform \eqref{T< Sigma1-Sigma2}
into
\begin{align}
T_{\mu/\nu}(x,y)\le&\, 2 - (1+O(1-x))\frac{2}{1+\sqrt{1-x^{\mu/\nu}}}\notag\\
=&\,\frac{2\sqrt{1-x^{\mu/\nu}}}{1+\sqrt{1-x^{\mu/\nu}}}\cdot (1+O(\sqrt{1-x}\,))\label{Tnumu<2sqrt}.
\end{align}
Using \eqref{Tnumu<2sqrt} we replace \eqref{PAnumu<} with
\begin{align}
&\pr(A_{\nu,\mu}) \le_b\, n^4 \bigl[R_{\nu,\mu}+o(1)\bigr]^{\nu}/ \nu^2 ,\label{(Rnumu+o1)nu}\\
&R_{\nu,\mu}:=\frac{4(1+\mu/\nu)^{1+\mu/\nu}\,(m/n)^{\mu/\nu}}{(\mu/\nu)^{\mu/\nu}\,(1+m/n)^{1+\mu/\nu}}
\cdot \left(\frac{\sqrt{1-x^{\mu/\nu}}}{1+\sqrt{1-x^{\mu/\nu}}}\right)^2.\label{Rnumu=}
\end{align}
Define $X=\tfrac{\mu/\nu}{m/n}$. Here, since  $\mu/\nu\ge m/n\to\infty$,
\begin{multline}\label{4Xe1-X}
\frac{4(1+\mu/\nu)^{1+\mu/\nu}}{(\mu/\nu)^{\mu/\nu}}\cdot\frac{(m/n)^{\mu/\nu}}
{(1+m/n)^{1+\mu/\nu}}\\
=4X\bigl(1+1/(\mu/\nu)\bigr)^{1+\mu/\nu}\cdot\bigl(1-1/(1+m/n)\bigr)^{1+\mu/\nu}\\
=4Xe^{1+o(1) -X(1+o(1))},
\end{multline}
uniformly over $X$. We have two cases.

{\bf Case  $X\ge c$.\/}  Since
\[
\frac{\sqrt{1-x^{\mu/\nu}}}{1+\sqrt{1-x^{\mu/\nu}}}\le \frac{1}{2},
\]
we have 
\begin{equation*}
R_{\nu,\mu}\le Xe^{1+o(1) -X(1+o(1))}\le \rho+o(1),\quad\rho:=ce^{1-c}<1,
\end{equation*}
as $c>1$. Thus,
\[
\pr(A_{\nu,\mu})\le_b\, n^4 \bigl[R_{\nu,\mu}+o(1)\bigr]^{\nu} \le n^4(\rho+o(1))^{\nu}
\]
and
\begin{align*}
\sum_{\nu,\mu: X\ge c} \pr(A_{\nu,\mu}) \le_b&\, \sum_{\nu\ge \alpha \log n} \sum_{\mu\ge \nu-1}n^4(\rho+o(1))^{\nu}\\
 \le&\, n^6\sum_{\nu} (\rho+o(1))^{\nu} \le_b  n^6 (\rho+o(1))^{\alpha \log n} \to 0
 \end{align*}
since $\alpha\ge 7/\log(1/\rho)$.

{\bf Case  $X\le c$.\/} The function $\phi(z)=\frac{\sqrt{1-z}}{1+\sqrt{1-z}}$ is decreasing on $(0,1)$, so to find an upper bound for $\phi(x^{\mu/\nu})$, we  want to bound $x^{\mu/\nu}$ from below. We have
\begin{align*}
x^{\mu/\nu}=&\,\exp\bigl[-(\mu/\nu)(1-x)+O((\mu/\nu)(1-x)^2)\bigr]\\
=&\,\exp\bigl[-(\mu/\nu)(1-x)+O(n/m)\bigr].
\end{align*}
Further, using 
\[
\mu/\nu-m/n \le (c-1)(m/n),\quad m-\mu+n-\nu\ge m-\mu \ge (1-c)m,
\]
we compute
\begin{align*}
-\frac{\mu}{\nu}(1-x)&=\,-\frac{\mu}{\nu}\cdot\frac{n-\nu}{m-\mu+n-\nu}\\
&=\,-X-\frac{\mu}{\nu}\left(\frac{n-\nu}{m-\mu+n-\nu}-\frac{n}{m}\right)\\
&=\,-X-\frac{\mu}{\nu}\cdot\frac{n\nu(\mu/\nu - m/n)-n(n-\nu)}{(m-\mu+n-\nu)m}\\
&\ge \,-X-\frac{\mu}{\nu}\cdot\frac{n\nu(\mu/\nu - m/n)}{(m-\mu+n-\nu)m} \\
&\ge \,-X-\frac{\mu}{\nu}\cdot\frac{n\nu(c-1)(m/n)}{(c-1)m^2} \\
&= \, -X - \mu/m \ge -c-(2-c)=-2.
\end{align*}
Consequently,
$ x^{\mu/\nu} \ge e^{-3}$, and $\phi(x^{\mu/\nu})\le \phi(e^{-3}) \le 0.494$.
Since $Xe^{1-X}$ is decreasing on $(1,\infty)$ and takes the value $1$ for $X=1$, we have 
\[
Xe^{1+o(1) -X(1+o(1))} \le 1+o(1)
\]
for $1\le X\le c$. Therefore,
\[
R_{\nu,\mu} \le 4\times (0.495)^2 \le 0.981
\]
and
\[
\pr(A_{\nu,\mu})\le_b\, n^4 \bigl[R_{\nu,\mu}+o(1)\bigr]^{\nu} \le n^4 (0.99)^{\nu} 
\]
As in the previous case,
\begin{align*}
\sum_{\nu,\mu: X\le c} \pr(A_{\nu,\mu})& \le_b \sum_{\nu} \sum_{\mu}n^4(0.99)^{\nu}\\
& \le n^6\sum_{\nu} (0.99)^{\nu} \le_b  n^6 (0.99)^{\alpha \log n} \to 0,
\end{align*}
since $\alpha>7\log (1/0.99)$. 
\end{proof}

Letting $c\downarrow 1$, we arrive at 
\begin{Corollary}\label{cor1} Whp,
\begin{itemize}[itemsep=0pt, parsep=0pt, topsep=0pt, partopsep=0pt]
\item either the densest component is of size $O(\log n)$, 
\item or its number of crossings is almost $m$, whence its size is at least $\sqrt{2m}$.
\end{itemize}
\end{Corollary}

\begin{Remark} This is a good place to notice that the sole reason for $\log n$ to appear in the
first alternative was that we confined ourselves to $m=O(n\log n)$ meeting the 
constraint \eqref{m<},
in which case $T_{n,m}$ is bounded from below by $C_n\binom{n+m-1}{n-1}\exp(-\gamma
\log n)$. For the constraint $n\log n\ll m \ll n^{3/2}$ we still have the lower bound \eqref{lowersimple}
from Lemma \ref{bound,allmn}, 
\[
T_{n,m}\ge_b \exp\bigl(-\Theta((m/n)\log n)\bigr) C_n\binom{n+m-1}{n-1}.
\]
To off-set this exponential factor,  we could have confined ourselves to $\nu$ of order 
$(m/n)\log n$,
at least, arriving at the counterpart of Corollary \ref{cor1} with the first alternative becoming
``either the densest component is of size $O((m/n)\log n)$'', but with the second alternative
remaining unchanged. In other words, the gap property for the crossing density of the densest 
component continues to hold for $n\log n \ll m\ll n^{3/2}$. 
\end{Remark}
 
Now if $m=\Theta(n\log n)$, and the densest component
has size $\nu$ then for the number of crossings we have
\[
\frac{\nu(\nu-1)}{2}\ge \mu \ge \nu \frac{m}{n}\Longrightarrow \nu \ge 2 m/n=\Theta(\log n).
\]
So, if $\nu=O(\log n)$, then $\nu=\Theta(\log n)$ and $\mu=\Theta((\log n)^2)$, and the
maximum density $\mu/\nu$ is of order $m/n$ exactly. That's the reason why in  the rest
of the paper we continue to stick with $m=\Theta(n\log n)$.

\begin{Lemma}\label{mu/m<1} 
Given fixed $c\ge 1$, $b>1$, let $B_{n,m}=B_{n,m}(c,b)$ denote the event:  the maximum density is below $c\,m/n$ and 
there is a $(\nu,\mu)$-component meeting the constraints
\begin{equation}\label{defBnm}
\nu\ge b \log n,\quad  \mu\le (1-b^{-1/3})m.
\end{equation}
For every $c\ge 1$, there exists $b=b(c)>1$ such that $\pr(B_{n,m})\to 0$. 
\end{Lemma}
\begin{proof} First of all, in view of Lemma \ref{densitygap}, by choosing $b$ sufficiently
large we can consider only  $(\nu,\mu)$-components with $\mu/\nu\ge d m/n$, with some
fixed $d>0$. Also, for $\mu$  satisfying  \eqref{defBnm},
\[
\frac{m-\mu}{n-\nu}\ge \frac{mb^{-1/3}}{n}=\Theta(b^{-1/3}\log n)\to\infty.
\]
Arguing as in the proof of Lemma \ref{maxden}, we obtain
\begin{equation}\label{one}
\pr(B_{n,m})\le_b n^4\sum_{\nu,\mu}\bigl[R_{\nu,\mu}(1+O(b^{-1}))\bigr]^{\nu},
\end{equation}
where the sum is over all $(\nu,\mu)$ satisfying \eqref{defBnm}, but instead of \eqref{Rnumu=} we get 
\begin{align*}
R_{\nu,\mu}:=&\,4\frac{(1+\mu/\nu)^{1+\mu/\nu}}{(\mu/\nu)^{\mu/\nu}}\cdot\frac{(m/n)^{\mu/\nu}}
{(1+m/n)^{1+\mu/\nu}}
\cdot \left(\frac{\sqrt{1-x^{cm/n}}}{1+\sqrt{1-x^{cm/n}}}\right)^2\\
\le&\, \frac{(1+\mu/\nu)^{1+\mu/\nu}}{(\mu/\nu)^{\mu/\nu}}\cdot\frac{(m/n)^{\mu/\nu}}
{(1+m/n)^{1+\mu/\nu}}.
\end{align*}
Here as before 
\begin{equation}\label{two}
x=(m-\mu)/(m-\mu+n-\nu)=1-O(b^{-1}).
\end{equation}
(The remainder $O(b^{-1})$ in
\eqref{two} is the reason for the same remainder in \eqref{one}.)  Again, set $X=\tfrac{\mu/\nu}{m/n}$. Since  $m/n\to\infty$ and $\mu/\nu\to\infty$,
\begin{equation*}
R_{\nu,\mu}(1+O(b^{-1}))\le Xe^{1-X+O(b^{-1})},
\end{equation*}
The log-concave function $H(X):=Xe^{1-X}$ attains its absolute maximum $1$ at $X=1$. Let $A>0$ be a constant and first consider the contribution of $X$'s with $|X-1|\ge Ab^{-1/2}$.
We have
\begin{multline*}
\max\{H(X): |X-1|\ge A b^{-1/2}\}\\
\le\,\max\bigl\{H(1-Ab^{-1/2}),H(1+Ab^{-1/2})\bigr\}
\le\exp\bigl[-A^2/(3b)\bigr].
\end{multline*}
Thus, for this range of $X$,
\[
R_{\nu,\mu}(1+O(b^{-1}))\le \exp\bigl[-A^2/(4b)\bigr]
\]
if we choose $A$ sufficiently large. So 
\begin{equation*}
 n^4\cdot \!\!\!\!\!\!\!\!\sum_{\nu\ge b\log n\atop \left|X-1\right|\ge Ab^{-1/2}}
 \!\!\!\!\!\!\!\!\!\bigl[R_{\nu,\mu}(1+O(b^{-1}))\bigr]^{\nu}
 \le\, n^4\cdot \!\!\!\!\sum_{\nu\ge b\log n}\!\!\nu^2\exp\bigl[-\nu A^2/(4b)\bigr]\to 0,
 \end{equation*}
if $b(A^2/(4b))>5$, that is, if  $A^2>20$. The factor $\nu^2$ in the second sum is due to the fact that there are at most ${\nu \choose 2}$ values of $\mu$.

Now consider the contribution of $(\nu,\mu)$ where $\left|X-1\right|\le  Ab^{-1/2}$. 
We have
\[
x^{cm/n}=\exp\bigl[-(1-x)cm/n +O((1-x)^2 m/n)\bigr]
\]
and
\begin{align*}
(1-x)\frac{m}{n}=&\,\frac{m}{m+n}\left[1+\frac{\nu(\mu/\nu-m/n)}{m-\mu+n-\nu}\right]\\
=&\,\frac{m}{m+n}\bigl[1+O(\mu b^{-1/2}/(m-\mu))\bigr]\\
=&\,\frac{m}{m+n}\bigl[1+O(b^{-1/6})\bigr] =\, 1 +o(1).
\end{align*}
Therefore, introducing $\rho=2\tfrac{\sqrt{1-e^{-c}}}{1+\sqrt{1-e^{-c}}}<1$, we obtain
\[
R_{\nu,\mu}(1+O(b^{-1}))\le \rho \bigl(1+O(b^{-1/6})\bigr)Xe^{1-X}\le \rho^{1/2},
\]
if $b$ is large enough. We conclude that
\begin{align*}
 n^4\!\!\!\!\!\!\!\!\sum_{\nu\ge b\log n\atop |X-1|\le Ab^{-1/2}}
\!\!\!\!\!\!\!\!\! \bigl[R_{\nu,\mu}(1+&O(b^{-1}))\bigr]^{\nu}
 \le \,n^4\!\!\!\!\sum_{\nu\ge b\log n}\!\!\nu^2(\rho^{1/2})^{\nu}\to 0,
 \end{align*}
if  $b$ is sufficiently large.
\end{proof}

\begin{Corollary}\label{cor2} 
Suppose that $\lim_{n\to\infty} m/(n\log n)\in (0,2/\pi^2)$. Whp, 
\begin{itemize}[itemsep=0pt, parsep=0pt, topsep=0pt, partopsep=0pt]
\item either there exists a (necessarily unique)  component  that contains almost all $m$ crossings, whence has at least $\sqrt{2m}$ vertices,
\item or there is no component of size $\nu$ with $\nu/\log n$ exceeding a large constant.
\end{itemize}
\end{Corollary}

\begin{proof}
It follows directly from Lemma \ref{maxden} and Lemma~\ref{mu/m<1}.
\end{proof}

The preceding analysis was based on the bound \eqref{uppersimple}, which was implied by $T_{n,m}\le C_n I_{n,m}$.
Given $k ,\ell>1$ and $s\le k$, let $T_{n,m}(k,\ell,s)$ denote the total number of diagrams with $k$ components, each of size not exceeding $\ell$, and with exactly $s$ components of size $1$,
i.e., isolated chords.  Obviously, 
\begin{equation}\label{Tnmkells=0?}
T_{n,m}(k,\ell,s)=0\quad \text{ if } \quad \ell (k - s) < n-s.
\end{equation}
\begin{Lemma}\label{Tnm(k,ell,s)<} Introducing
\[
I_j(x):=\sum_{\mu\ge 0} I_{j,\mu} x^{\mu}=(1+x)\cdots(1+x+\cdots+x^{j-1}),
\]
we have: 
\begin{equation}\label{TNm(k,ell)Ijx}
T_{n,m}(k,\ell,s)\le \frac{(2n)_{k-1}}{(k-s)!s!}\,[x^my^{n-s}]\left(\sum_{j=2}^{\infty}C_jI_j(x)y^j\right)^{k-s}.
\end{equation}
\end{Lemma}
\begin{proof} For a generic diagram with parameters $n$ and $m$, with $k$ components, of
size not exceeding $\ell$, and $s$ components of size $1$,
let $s_j$ denote the total number of components of size $j$; so $\bs s=(s_1,s_2,\dots,s_n)$
meet the conditions: 
\begin{equation}\label{smeet}
s_1=s;\quad (\forall j>\ell) \  s_j=0; \quad\sum_{j=2}^n s_j=k-s;\quad \sum_{j=2}^njs_j=n-s.
\end{equation}
For such a diagram to exist, it is necessary that the point sets of the components form
a non-crossing partition of $[2n]$. By Kreweras's formula \cite{Kreweras}, the total number of
such partitions is $(2n)_{k-1}/[s_1!s_2!\cdots]$. In addition, for each $2\le j\le n$, and $1\le t\le s_j$, let $m_{j,t}$ denote the number of crossings of the $t$-th component from the arbitrarly ordered
list of all components of size $j$. Clearly, $\bs m= \{m_{j,t}\}$ meets the condition
\begin{equation}\label{mjtmeet}
\sum_{j=2}^n\sum_{t=1}^{s_j}m_{j,t}=m.
\end{equation}
Then,
\allowdisplaybreaks{
\begin{align}
T_{n,m}(k,\ell,s)\le&\,\sum_{\bs s\text{ meets }\eqref{smeet}}\frac{(2n)_{k-1}}{s_1!\cdots s_n!}
\sum_{\bs m\text{ meets }\eqref{mjtmeet}}\prod_{2\le j\le n\atop1\le t\le s_j}T_{j,m_{j,t}}\notag\\
\le&\,\frac{1}{s!}\sum_{\bs s\text{ meets }\eqref{smeet}}\frac{(2n)_{k-1}}{s_2!\cdots s_n!}
\sum_{\bs m\text{ meets }\eqref{mjtmeet}}\prod_{2\le j\le n\atop 1\le t\le s_j} C_jI_{j,m_{j,t}}\notag\\
=&\,\frac{(2n)_{k-1}}{s!}\sum_{\bs s\text{ meets }\eqref{smeet}}\frac{1}{s_2!\cdots s_n!}\,\,
[x^m]\prod_{j=2}^n \left(C_j\sum_{\mu\ge 0}I_{j,m}x^{\mu}\right)^{s_j}\notag\\
=&\,\frac{(2n)_{k-1}}{s!}\,\, [x^m]\sum_{\bs s\text{ meets }\eqref{smeet}}\prod_{j=2}^n\frac{(C_jI_j(x))^{s_j}}
{s_j!}.\label{Tnm(k,ell)=[xm]}
\end{align}
}
Here the last sum is at most 
\begin{multline}\label{lastsum=[yn]}
[y^{n-s}z^{k-s}] \sum_{\bs s\ge \bs 0\atop \sum\limits_{j\ge 2} js_j<\infty}\prod_{j=2}^\infty\frac{(y^jzC_jI_j(x))^{s_j}}{s_j!}
=\,[y^{n-s}z^{k-s}]\exp\left(z\sum_{j\ge 2}y^jC_jI_j(x)\right) 		\\
=\,[y^{n-s}]\,\frac{1}{(k-s)!}\left(\sum_{j=2}^{\infty}y^jC_jI_j(x)\right)^{k-s}.
\end{multline}
Equations \eqref{Tnm(k,ell)=[xm]} and \eqref{lastsum=[yn]} imply \eqref{TNm(k,ell)Ijx}, which finishes the proof.
\end{proof}

Using 
\[
(1+x)\cdots(1+x+\cdots+x^{j-1})=(1-x)^{-j} (1-x)\cdots(1-x^j)
\]
we get:  for $k-s\ge (n-s)/\ell$,
\begin{equation*}
T_{n,m}(k,\ell,s)\le\frac{(2n)_{k-1}}{(k-s)!s!} \,[x^{m} y^{n-s}]\left(\sum_{j=2}^{\infty}C_j\left(\frac{y}{1-x}\right)^j
\prod_{t=1}^j(1-x^t)\right)^{k-s}.
\end{equation*}
The bivariate series on the RHS has positive coefficients, and converges for $x\in (0,1)$,
$y\in (0, (1-x)/4]$. So, by Chernoff-type bound with $x\in (0,1)$ and $y=(1-x)/4$, we obtain
\begin{align*}
T_{n,m}(k,\ell,s)\le&\,\frac{(2n)_{k-1}}{(k-s)!s!} \,x^{-m} y^{-(n-s)}\left((1-x)(1-x^2)\sum_{j=2}^{\infty}\frac{C_j}{4^j}
\prod_{t=3}^j(1-x^t)\right)^{k-s}\\
\le&\,\frac{4^{n-s}(2n)_{k-1}}{(k-s)!s!} \,x^{-m} (1-x)^{-(n+s-2k)}\bigl[1/4 +O(1-x)\bigr]^{k-s}\\
=&\,\frac{4^{n-k}(2n)_{k-1}}{(k-s)!s!}\, x^{-m} (1-x)^{-(n+s-2k)}\bigl[1 +O(1-x)\bigr]^{k-s}.
\end{align*}
Choosing $x=m/(m+n)$ we get
\begin{equation}\label{Tnmkells<explicit}
\begin{aligned}
T_{n,m}(k,\ell,s)\le&\, {k \choose s}\frac{4^{n-k} (2n)_{k-1}}{k!} \\
&\times\frac{(m+n)^{m+n+s-2k}}{m^m n^{n+s-2k}}\,\bigl[1 +O(n/(m+n))\bigr]^{k-s}.
\end{aligned}
\end{equation}
\begin{Theorem}\label{almostm,m=nlogn} 
Suppose that $\lim m/(n\log n)\in (0,2/\pi^2)$. Then whp
there exists a component that has almost all $m$ crossings.
\end{Theorem} 
\begin{proof} First of all, by Corollary \ref{cor2}, it suffices to prove that whp
there is a  component of size exceeding $\ell:=A \log n$. Let $X$ denote the total number of isolated chords in the random diagram. Clearly
\begin{equation}\label{T(x,y)2}
\begin{aligned}
\ex[X]\le&\, \frac{2n}{T_{n,m}}\cdot\sum_{n_1+n_2=n-1\atop m_1+m_2=m}T_{n_1,m_1}T_{n_2,m_2}\\
=&\,\frac{2n}{T_{n,m}}\,[x^my^{n-1}]\,T(x,y)^2.
\end{aligned}
\end{equation}
So, using \eqref{asym,ext,ell} with $\ell=2$, 
\begin{equation*}
\ex[X]\le_b n\frac{\binom{n+m-2}{n-2}}{\binom{n+m-1}{n-1}}=\frac{n(n-1)}{n+m-1}=O(n(\log n)^{-1}).
\end{equation*}
So whp $X\le n/(\log n)^{1-\eps}$, if $\eps\in (0,1)$ is fixed. Thus, it suffices to prove that
\begin{equation}\label{suffices}
\sum_{k,s}\frac{T_{n,m}(k,\ell,s)}{T_{n,m}}\to 0,
\end{equation}
where the sum is over all $k,s$ such that
\begin{equation}\label{s<,k-s>}
s\le s(n): \frac{n}{(\log n)^{1-\eps}},\qquad k-s \ge \frac{n-s}{\ell}.
\end{equation}
(Indeed, $T_{n,m}(k,\ell,s)/T_{n,m}$ is the probability that the diagram has $k$ components, 
with exactly $s$ components of size $1$, and all other components of size {\it not 
exceeding\/} $\ell$.) Combining the asymptotic formula \eqref{asym,ext,ell} for $T_{n,m}$ in Lemma \ref{ell components}, 
the bound \eqref{Tnmkells<explicit} for $T_{n,m}(k,\ell,s)$, and the constraints 
\eqref{s<,k-s>} we obtain: 
\begin{align*}
\frac{T_{n,m}(k,\ell,s)}{T_{n,m}}\ll&\, \binom{2n}{k}\left(\frac{n}{m+n}\right)^{k+n/(2\ell)}\\
\le&\, \frac{(2n)^k}{k!}\left(\frac{n}{m}\right)^k\times \left(\frac{n}{m}\right)^{n/(2\ell)}\\
\le&\,\exp\bigl(2n^2/m \bigr)\cdot\exp\left(-\frac{n}{2A\log n}\log(m/n)\right)\\
\le&\,\exp\left(O(n/\log n)-\gamma\,\frac{n\log\log n}{\log n}\right),
\end{align*}
$\gamma>0$ being fixed. For the third line in the above inequality, we used $y^k/k! \le e^y$. The last quantity approaches $0$ super-polynomially fast. Therefore, so does the expression in \eqref{suffices}.
\end{proof}

\begin{Theorem}\label{almostm,n;m=nlogn} Suppose that $\lim m/(n\log n)\in (0,2/\pi^2)$. Then whp
there exists a component that has almost all $m$ crossings and a positive fraction of $n$ chords.
\end{Theorem} 
\begin{proof} Given $\eps,\,\delta\in (0,1)$, let $\mathcal{C}_{\eps,\delta}$ denote the total number of the $(\nu,\mu)$-components with $\nu\le \delta n$ and $\mu\ge (1-\eps)m$. In light of Theorem \ref{almostm,m=nlogn}, it suffices to show that, for $\eps<1/2$, and $\delta$ sufficiently small, 
$\ex[\mathcal{C}_{n,\eps}]\to 0$. Let $\delta<1$ be such that
\begin{equation}\label{delta<}
\delta<\frac{(1/2-e^{-1})(1-\eps)}{\log (4e)}.
\end{equation}
By \eqref{EXnumu,cruder}
\begin{align*}
\ex[\mathcal{C}_{\eps,\delta}]\le_b&\  \sum_{\sqrt{2\mu}\le \nu\le \delta n\atop \mu\ge
(1-\eps)m} (n^4/\nu^2)\exp\bigl[\nu H(x_{\nu,\mu})\bigr],\quad x_{\nu,\mu}:=\frac{\mu}{\nu};\\
H(x):=&\,\log 4+(1+x)\log\frac{1+x}{1+m/n}+x\log\frac{m/n}{x}.
\end{align*}
Observe that, for $\nu,\,\mu$ in question,
\[
x_{\nu,\mu}\ge \frac{m}{n}\,\frac{1-\eps}{\delta}>y:=\frac{m}{n},
\]
since $1-\eps>\delta$. Further,
\begin{align*}
H(x)= &\, \log 4 +(1+x)\log\frac{x}{y}+(1+x)\log\frac{1+1/x}{1+1/y}+x\log\frac{y}{x}\\
\le&\, \log 4 +\log\frac{x}{y}+(1+x)\left(\frac{1+1/x}{1+1/y}-1\right)\\
\le&\, \log 4 +1 +\log\frac{x}{y}-\frac{x/y}{1+1/y}\\
\le&\, \log(4e) +\log\frac{x}{y}-\frac{1}{2}\,\frac{x}{y}\\
\le&\, \log(4e)-(1/2 - e^{-1})\,\frac{x}{y},
\end{align*}
the last inequality following from $\log z \le e^{-1}z$. Therefore
\begin{align*}
H(x_{\nu,\mu})\le&\, -\gamma(\eps,\delta),\\
\gamma(\eps,\delta):=&\,(1/2 - e^{-1})\frac{1-\eps}{\delta}-\log(4e)>0,
\end{align*}
see \eqref{delta<}. Therefore, as $n\to\infty$,
\begin{equation*}
\ex[\mathcal{C}_{\eps,\delta}]\le_b n^4\sum_{\nu\ge \sqrt{m}} \exp\bigl[-\nu \gamma(\eps,
\delta)\bigr]\to 0. \qedhere
\end{equation*}
\end{proof}

To complete the picture, turn now to $m=\Theta(n)$.
\begin{Theorem}\label{mordern}
If $m\le n/14$, then there exists a constant $A>0$ such
that whp the size of the largest component is at most $A \log n$.
\end{Theorem}
\begin{proof} Let $A>0$ to be specified shortly. Then for $E_n$, the expected number of 
components of size exceeding $A \log n$,  (by \eqref{EXnumu,cruder} again), we have
\[
E_n\le_b \ n^4\sum_{\nu\ge A\log n}\,\sum_{\mu\ge \nu-1}\nu^{-2}\cdot \exp\bigl[\nu H(x_{\nu,\mu})\bigr];
\]
here $x_{\nu,\mu}=\mu/\nu \ge 1 - 1/(A \log n)$. Since $H(x)$ is concave, 
\begin{align*}
H(x_{\nu,\mu})\le&\, H(1) +H^\prime(1)(x_{\nu,\mu}-1)\\
\le&\, \log\frac{16m/n}{(1+m/n)^2} +O((\log n)^{-1}).
\end{align*}
Now $16z/(1+z)^2<1$ for $0<z<z^*:=7-\sqrt{48}> 1/14$. So if $m/n\le 1/14$, then 
\begin{equation*}
E_n\le_b n^4m\sum_{\nu\ge A \log n} \exp\left[\nu\left(\log\frac{224}{225}+O((\log n)^{-1})\right)\right]
\to 0,
\end{equation*}
if $A> 5/ (\log 225/224)$.
\end{proof}

\section{Concluding Remarks}

Although chord diagrams have been studied widely, there are still many open problems about them, particularly of enumerative-probabilistic nature. The results presented in this paper provide partial solutions to some, in our opinion interesting, problems. We conclude this paper  with some questions for possible extensions of our results. 

An asymptotic expression for the number $T_{n,m}$ of chord diagrams with a given number of crossings has been found in Theorem \ref{ell components} for the case $m<(2/\pi^2)n \log $, but its extension
for larger $m$ is still to be found. It would be quite useful to even have usable lower and upper bounds for $T_{n,m}$.  Lemma~\ref{bound,allmn} gives an upper bound for all $n$ and $m$, but
a lower bound only for $m=o(n^{3/2})$. 

Our main goal in this paper was to observe a kind of phase transition for the largest component of a random chord diagram. Theorem~\ref{almostm,n;m=nlogn} tells us that when $m/n\log n$ has a limit in $(0,2/\pi^2)$, there is a giant component containing almost all the crossings (edges in the intersection graph) and a positive fraction of chords. In Erd\H{o}s-R\'{e}nyi graphs, coupling $G(n,m)$ with $G(n,m+1)$ with a graph process yields immediately that having a giant component is a monotone property. Finding a similar coupling for chord diagrams would imply the existence of a giant component (whp) for $m=\Omega(n\log n)$. On the other hand, it is still unclear whether there is a giant component or not for smaller values of $m$. For $m\le n/14$, this possibility is ruled out by Theorem~\ref{mordern}. Thresholds for various other graph theoretic properties  of random chord diagrams are also of interest to us.

Lastly, there are two other classes of graphs nontrivially related to chord diagrams: circle graphs and interlace graphs. A (labeled) circle graph is obtained by labeling a set of chords of a circle, where the edges are determined by the crossing relation. An \emph{interlace graph} with vertex set $[n]$  is obtained from a permutation of the multiset $\{1,1,2,2,\dots,n,n\}$, where two vertices $i$ and $j$ are adjacent if the corresponding symbols are interlaced in the permutation, i.e., if the permutation looks like $\dots i\dots j \dots i\dots j \dots$ or $\dots j\dots i \dots j\dots i\dots$ As Arratia et al.\ \cite{Arratia} pointed out, each circle graph is an interlace graph, and the number of interlace graphs is bounded above by the number of permutations of the multiset, which is $(2n)!/2^n$. A chord diagram corresponds to a \emph{standard permutation}, in which the first occurence of $i$ is always before the first occurrence of $j$ for all pairs $i<j$. However, the number of interlace graphs of standard permutations is not the same as the number of intersection graphs due to the fact that the same interlace graph might come from many different standard permutations, whereas the intersection graphs that we consider uniquely determine the chord diagrams. For example, there are $\tbinom{2n}{n}/(n+1)$ standard permutations producing the empty graph on $[n]$. 
We are curious if the results in this paper hold for these two important classes of graphs, or at least   
shed some light on the respective thresholds for the appearance of a giant component.

\section*{Acknowledgements}
The authors are thankful to Sergei Chmutov for encouraging us to study random chord diagrams and for helpful discussions.

\end{document}